\documentclass{amsart}
\usepackage{amssymb}
\usepackage{mathrsfs}
\usepackage[all]{xy}

\usepackage{graphicx}
\usepackage{caption}
\usepackage{subcaption}

\newcommand{\C}{\mathbb{C}}
\newcommand{\Ql}{\mathbb{Q}_\ell}
\newcommand{\Fl}{\mathbb{F}_\ell}
\newcommand{\Qlb}{\overline{\mathbb{Q}}_\ell}
\newcommand{\Z}{\mathbb{Z}}
\newcommand{\N}{\mathbb{N}}

\newcommand{\K}{\mathbb{K}}
\newcommand{\F}{\mathbb{F}}
\renewcommand{\O}{\mathbb{O}}
\newcommand{\E}{\mathbb{E}}
\newcommand{\bk}{\Bbbk}

\newcommand{\fg}{\mathfrak{g}}
\newcommand{\fz}{\mathfrak{z}}
\newcommand{\fp}{\mathfrak{p}}
\newcommand{\fl}{\mathfrak{l}}
\newcommand{\fu}{\mathfrak{u}}
\newcommand{\fh}{\mathfrak{h}}
\newcommand{\cN}{\mathscr{N}}
\newcommand{\cO}{\mathscr{O}}
\newcommand{\fL}{\mathfrak{L}}
\newcommand{\xymap}{\varpi}

\newcommand{\Db}{D^{\mathrm{b}}}
\newcommand{\Perv}{\mathsf{Perv}}
\newcommand{\cA}{\mathscr{A}}
\newcommand{\Rep}{\mathsf{Rep}}

\newcommand{\For}{\mathrm{For}}
\newcommand{\bT}{\mathbb{T}}
\newcommand{\Res}{\mathbf{R}}
\newcommand{\Ind}{\mathbf{I}}

\newcommand{\uInd}{\underline{\Ind}}

\newcommand{\cL}{\mathcal{L}}
\newcommand{\IC}{\mathcal{IC}}
\newcommand{\cF}{\mathcal{F}}
\newcommand{\cG}{\mathcal{G}}
\newcommand{\cH}{\mathcal{H}}
\newcommand{\cI}{\mathcal{I}}
\newcommand{\ubk}{\underline{\bk}}
\newcommand{\Spr}{\underline{\mathsf{Spr}}}
\newcommand{\cE}{\mathcal{E}}
\newcommand{\cS}{\mathcal{S}}
\newcommand{\cD}{\mathcal{D}}

\newcommand{\simto}{\xrightarrow{\sim}}
\DeclareMathOperator{\End}{End}
\DeclareMathOperator{\Hom}{Hom}
\DeclareMathOperator{\Ext}{Ext}
\DeclareMathOperator{\rad}{rad}
\DeclareMathOperator{\Irr}{Irr}
\newcommand{\id}{\mathrm{id}}

\makeatletter
\def\lotimes{\@ifnextchar_{\@lotimessub}{\@lotimesnosub}}
\def\@lotimessub_#1{\mathchoice{\mathbin{\mathop{\otimes}^L}_{#1}}%
  {\otimes^L_{#1}}{\otimes^L_{#1}}{\otimes^L_{#1}}}
\def\@lotimesnosub{\mathbin{\mathop{\otimes}^L}}
\makeatother

\newcommand{\GL}{\mathrm{GL}}
\newcommand{\fgl}{\mathfrak{gl}}
\newcommand{\SL}{\mathrm{SL}}
\newcommand{\fS}{\mathfrak{S}}

\newcommand{\la}{\lambda}
\newcommand{\Psico}{\Psi^{\mathrm{co}}}
\newcommand{\psico}{\psi^{\mathrm{co}}}
\newcommand{\Part}{\mathrm{Part}}
\newcommand{\Comp}{\N^\infty}
\newcommand{\sa}{\mathsf{a}}
\newcommand{\sfb}{\mathsf{b}}
\newcommand{\sm}{\mathsf{m}}
\newcommand{\uPart}{\underline{\Part}}
\newcommand{\bL}{\mathbf{L}}
\newcommand{\blambda}{{\boldsymbol{\lambda}}}
\newcommand{\tr}{\mathrm{t}}

\newtheorem*{thm*}{Theorem}
\numberwithin{equation}{section}
\newtheorem{thm}{Theorem}[section]
\newtheorem{lem}[thm]{Lemma}
\newtheorem{prop}[thm]{Proposition}
\newtheorem{cor}[thm]{Corollary}
\theoremstyle{definition}
\newtheorem{defn}[thm]{Definition}

\theoremstyle{remark}
\newtheorem{rmk}[thm]{Remark}

\DeclareFontFamily{U}{mathx}{\hyphenchar\font45}
\DeclareFontShape{U}{mathx}{m}{n}{
      <5> <6> <7> <8> <9> <10>
      <10.95> <12> <14.4> <17.28> <20.74> <24.88>
      mathx10
      }{}
\DeclareSymbolFont{mathx}{U}{mathx}{m}{n}
\DeclareFontSubstitution{U}{mathx}{m}{n}
\DeclareMathAccent{\widebar}{0}{mathx}{"73}

\makeatletter

\newlength{\tabwidth}
\newlength{\tabheight}
\newcommand{\tabstyle}{\textstyle}
\newlength{\tabrulel}
\newlength{\tabruler}
\newlength{\tabrulet}
\newlength{\tabruleb}
\newlength{\tabwidthx}
\newlength{\tabheightx}

\def\gentabbox#1#2#3#4#5#6#7{\vbox to \tabheight{%
  \setlength{\tabrulel}{#3}\setlength{\tabruler}{#4}%
  \setlength{\tabrulet}{#5}\setlength{\tabruleb}{#6}%
  \setlength{\tabwidthx}{#1\tabwidth}\addtolength{\tabwidthx}{0.5\tabrulel}%
    \addtolength{\tabwidthx}{0.5\tabruler}%
  \setlength{\tabheightx}{#2\tabheight}\addtolength{\tabheightx}{-\tabheight}%
  \hbox to #1\tabwidth{%
    \hspace{-0.5\tabrulel}\rule{\tabrulel}{#2\tabheight}\hspace{-\tabrulel}%
    \vbox to #2\tabheight{\hsize=\tabwidthx%
      \vspace{-0.5\tabrulet}\hrule width\tabwidthx height\tabrulet%
      \vspace{-0.5\tabrulet}\vss%
      \hbox to \tabwidthx{\hss#7\hss}%
        \vss\vspace{-0.5\tabruleb}%
      \hrule width\tabwidthx height\tabruleb\vspace{-0.5\tabruleb}}%
    \hspace{-\tabruler}\rule{\tabruler}{#2\tabheight}\hspace{-0.5\tabruler}}%
  \vspace{-\tabheightx}}}
\def\genblankbox#1#2{\vbox to \tabheight{\vfil\hbox to #1\tabwidth{\hfil}}}
\def\tabbox#1#2#3{\gentabbox{#1}{#2}{0.4pt}{0.4pt}{0.4pt}{0.4pt}{#3}}
\def\boldtabbox#1#2#3{\gentabbox{#1}{#2}{1.2pt}{1.2pt}{1.2pt}{1.2pt}{#3}}

\catcode`\:=13 \catcode`\.=13 \catcode`\;=13 
\catcode`\>=13 \catcode`\^=13
\def:#1\\{\hbox{$\tabstyle #1$}}
\def.#1{\tabbox{1}{1}{$\tabstyle #1$}}
\def>#1{\tabbox{2}{1}{$\tabstyle #1$}}
\def^#1{\tabbox{1}{2}{$\tabstyle #1$}}
\def;{\genblankbox{1}{1}\relax}
\catcode`\:=12 \catcode`\.=12 \catcode`\;=12 
\catcode`\>=12 \catcode`\^=7

\newenvironment{tableau}{\bgroup\catcode`\:=13 \catcode`\.=13
  \catcode`\;=13 \catcode`\>=13 \catcode`\^=13 
  \def\b##1##2##3{\boldtabbox{##1}{##2}{\vbox{##3}}}%
  \def\n##1##2##3{\tabbox{##1}{##2}{\vbox{##3}}}%
  \def\c{\tabbox{1}{1}{{}}}
  \vcenter\bgroup\offinterlineskip}{\egroup\egroup}

\makeatother

\title[Modular generalized Springer correspondence I]{Modular generalized Springer correspondence I: the general linear group}

\author{Pramod N. Achar}
\address{Department of Mathematics\\
  Louisiana State University\\
  Baton Rouge, LA 70803\\
  U.S.A.}
\email{pramod@math.lsu.edu}

\author{Anthony Henderson}
\address{School of Mathematics and Statistics\\
  University of Sydney, NSW 2006\\
  Australia}
\email{anthony.henderson@sydney.edu.au}

\author{Daniel Juteau}
\address{Laboratoire de Math\'ematiques Nicolas Oresme, Universit\'e de Caen Basse-Normandie, CNRS, 14032 Caen, France}
\email{daniel.juteau@unicaen.fr}

\author{Simon Riche}
\address{Universit{\'e} Blaise Pascal - Clermont-Ferrand II, Laboratoire de Math{\'e}matiques, CNRS, UMR 6620, Campus universitaire des C{\'e}zeaux, F-63177 Aubi{\`e}re Cedex, France
}
\email{simon.riche@math.univ-bpclermont.fr}


\thanks{P.A. was supported by NSF Grant No.~DMS-1001594.  A.H. was supported by ARC Future Fellowship Grant No.~FT110100504. D.J. was supported by ANR Grant No.~ANR-09-JCJC-0102-01. S.R. was supported by ANR Grants No.~ANR-09-JCJC-0102-01 and ANR-2010-BLAN-110-02.}

\begin{document}

\begin{abstract}
We define a generalized Springer correspondence for the group $\GL(n)$ over any field. We also determine the cuspidal pairs, and compute the correspondence explicitly. Finally we define a stratification of the category of equivariant perverse sheaves on the nilpotent cone of $\GL(n)$ satisfying the `recollement' properties, and with subquotients equivalent to categories of representations of a product of symmetric groups.
\end{abstract}

\maketitle

\section{Introduction}
\label{sec:intro}

\subsection{}
This paper is the first in a series devoted to constructing and describing a \emph{modular generalized Springer correspondence} for reductive algebraic groups, with a view to a future theory of modular character sheaves. In this part we lay the foundations of our approach, and we construct and describe explicitly the correspondence in the case $G=\GL(n)$.

\subsection{}
Let $G$ be a connected complex reductive group with Weyl group $W$, and let $\bk$ be a field.  The \emph{Springer correspondence} over $\bk$ is an injective map from the set of isomorphism classes of irreducible $\bk[W]$-modules to the set of isomorphism classes of simple $G$-equivariant perverse $\bk$-sheaves on the nilpotent cone $\cN_G$ for $G$:
\begin{equation}\label{eqn:ord-springer}
\Irr(\bk[W]) \;\hookrightarrow\; \Irr(\Perv_G(\cN_G,\bk)).
\end{equation}
For $\bk$ of characteristic zero, this correspondence was effectively defined by Springer in~\cite{springer}. For $\bk$ of positive characteristic, this \emph{modular Springer correspondence} was defined by the third author in~\cite{juteau}; see also~\cite[Corollary 5.3]{ahjr}.

In general, the map~\eqref{eqn:ord-springer} is not surjective, and one may seek to understand the objects missing from its image.  For $\bk$ of characteristic zero, a uniform explanation was given by Lusztig~\cite{lusztig}.  His \emph{generalized Springer correspondence} is a bijection
\begin{equation}\label{eqn:gen-springer}
\bigsqcup_{\substack{\text{$L \subset G$ a Levi subgroup} \\ \text{$\cF \in \Irr(\Perv_L(\cN_L,\bk))$ cuspidal}}} \Irr(\bk[N_G(L)/L]) \;\longleftrightarrow\; \Irr(\Perv_G(\cN_G,\bk)).
\end{equation}
(Here, the disjoint union is over $G$-conjugacy classes of pairs $(L,\cF)$. The definition of `cuspidal' simple perverse sheaf will be recalled in \S\ref{ss:cuspidal}.)  The original Springer correspondence~\eqref{eqn:ord-springer} is the part of~\eqref{eqn:gen-springer} where $L=T$ is a maximal torus. The generalized Springer correspondence was described explicitly for all groups by combining works of Alvis, Hotta, Lusztig, Shoji, Spaltenstein and Springer; see \cite[\S 6.7 and \S 12]{shoji} for examples and references.

It is natural to ask whether a
bijection such as~\eqref{eqn:gen-springer} still holds when $\bk$ has positive characteristic, that is, whether there is a \emph{modular generalized Springer correspondence}.
In the main results of this paper, Theorems~\ref{thm:main} and~\ref{thm:explicit}, we prove this property for the group $G = \GL(n)$
and give an explicit combinatorial description of the bijection in that case.

\subsection{}
The main reason for treating the case of $\GL(n)$ separately is that this case avoids most of the technicalities related to non-constant local systems and bad primes (cf.~the different behaviour of the modular Springer correspondence for other classical groups in characteristic~$2$~\cite{jls}). Nevertheless, it still displays the key features of the more general situation.

When $G=\GL(n)$, the set $\Irr(\Perv_G(\cN_G,\bk))$ is essentially independent of $\bk$ (it is in bijection with the set of partitions of $n$). In contrast, the Weyl group $\fS_n$ and the
other finite groups $N_G(L)/L$ may have fewer irreducible
representations in positive characteristic than in characteristic zero.  Our results mean that this is exactly compensated for by the existence of more cuspidal simple perverse sheaves.  Indeed, when $\bk$ has characteristic zero, the only Levi subgroups of $\GL(n)$ admitting a cuspidal simple perverse sheaf are the maximal tori, and~\eqref{eqn:ord-springer} is already a bijection.  When $\bk$ has characteristic $\ell>0$, any Levi subgroup of $\GL(n)$ isomorphic to a product of groups of the form $\GL(\ell^k)$ (with $k \geq 0$) admits a unique cuspidal simple perverse sheaf up to isomorphism; see Theorem~\ref{thm:cuspidals-GL-new}.

\subsection{}
A striking new phenomenon in the modular case is that the disjoint
union in~\eqref{eqn:gen-springer} is related to a nontrivial
stratification, or more precisely an iterated `recollement', of the category $\Perv_G(\cN_G,\bk)$.
This stratification generalizes
the fact that the category of finite-dimensional $\bk[W]$-modules can be realized as a
quotient of $\Perv_G(\cN_G,\bk)$; see~\cite{mautner} and~\cite[Corollary 5.2]{ahjr}.  More generally, larger Levi
subgroups in~\eqref{eqn:gen-springer} correspond to lower strata in
the stratification; for a precise statement, see Theorem~\ref{thm:recollement}. 

Note that this property is invisible when $\bk$ has characteristic $0$ because of the semisimplicity of the category $\Perv_G(\cN_G,\bk)$. However a more subtle property holds in this case, resembling the stratification statement: if two simple perverse sheaves correspond under~\eqref{eqn:gen-springer} to non-conjugate pairs $(L,\cF)$, then there are no nontrivial morphisms between them (of any degree) in the derived category of sheaves on $\cN_G$; see~\cite[Theorem 24.8(c)]{lusztig-cs}.

\subsection{}
In~\cite{lusztig}, Lusztig was able to define the generalized Springer correspondence before classifying cuspidal simple perverse sheaves.  His construction made use of the decomposition theorem, which is not available in the modular case. Our approach is therefore different: we classify cuspidal simple perverse sheaves and prove the correspondence at the same time, deducing the existence of certain cuspidal objects from Lusztig's characteristic-zero classification by modular reduction.

A further difference is that Lusztig worked on the unipotent variety in the group $G$ rather than on the nilpotent cone $\cN_G$, and allowed $G$ to be defined over a field of characteristic $p$ (with $\bk=\Qlb$ for $\ell\neq p$). In this paper the group $G$ is over $\C$, so its unipotent variety and nilpotent cone are ($G$-equivariantly) isomorphic. Following~\cite{lusztig-fourier, mirkovic}, we use the nilpotent cone because a crucial role in our arguments is played by the Fourier--Sato transform on the Lie algebra $\fg$.

From~\cite{juteau} and~\cite{mautner}, one knows that an object in $\Irr(\Perv_G(\cN_G,\bk))$ belongs to the image of the Springer correspondence~\eqref{eqn:ord-springer} if and only if its Fourier--Sato transform has dense support in $\fg$; if so, this Fourier--Sato transform is the intersection cohomology extension of a local system on the regular semisimple set, namely the local system determined by the corresponding irreducible $\bk[W]$-module. Our generalization, carried out in this paper when $G=\GL(n)$, amounts to describing the Fourier--Sato transform of a general object in $\Irr(\Perv_G(\cN_G,\bk))$ similarly, as the intersection cohomology extension of a local system on a stratum in the Lusztig stratification of $\fg$. The partition of $\Irr(\Perv_G(\cN_G,\bk))$ according to which stratum occurs is what corresponds to the disjoint union on the left-hand side of~\eqref{eqn:gen-springer}, and what gives rise to the stratification of the category $\Perv_G(\cN_G,\bk)$.

\subsection{}
After this work was completed we learnt that Mautner has a program of conjectures and partial results describing some properties of the equivariant derived category $\Db_G(\cN_G,\bk)$ and its subcategory $\Perv_G(\cN_G,\bk)$, which also adapts some aspects of Lusztig's generalized Springer correspondence to the modular setting. These results are quite different from ours; however they lead to an alternative proof of the classification of cuspidal simple perverse sheaves for $\GL(n)$ (Theorem \ref{thm:cuspidals-GL-new}). We thank Carl Mautner for explaining these results to us.

Finally, let us note that, as was pointed out to us by M.~Geck and G.~Malle, the combinatorics of our modular generalized Springer correspondence for $\GL(n)$ is reminiscent of the partitioning into Harish--Chandra series of irreducible unipotent representations of the finite group $\GL(n,q)$ in characteristic $\ell$, where $q\equiv 1\pmod \ell$. This could be regarded as evidence for a theory of modular character sheaves.

\subsection{Organization of the paper}
We begin in Section~\ref{sec:generalities} with a review of relevant
background on such topics as cuspidal perverse sheaves, induction and restriction functors, Fourier--Sato transform, and modular reduction.  In Section~\ref{sec:mainthm}
we prove the existence of the
bijection~\eqref{eqn:gen-springer} for $G = \GL(n)$ and give an
explicit combinatorial description of it.  The main results are stated as Theorems~\ref{thm:cuspidals-GL-new}, \ref{thm:main}, and~\ref{thm:explicit}.
Finally, the
aforementioned stratification of $\Perv_G(\cN_G,\bk)$ is proved in
Section~\ref{sec:recollement}.

\section{General background}
\label{sec:generalities}

Let $\bk$ be a field of characteristic $\ell \geq 0$. We consider sheaves with coefficients in $\bk$, but our varieties are over $\C$ (with the strong topology). For a complex algebraic group $H$ acting on a variety $X$, we denote by $\Db_H(X,\bk)$ the constructible $H$-equivariant derived category defined in~\cite{bl}, and by $\Perv_H(X,\bk)$ its subcategory of $H$-equivariant perverse $\bk$-sheaves on $X$. The constant sheaf on $X$ with value $\bk$ is denoted by $\ubk_X$, or simply $\ubk$.

The cases we consider most frequently are $X=\fh$ (the Lie algebra of $H$) and $X=\cN_H$ (the nilpotent cone of $H$). We usually make no notational distinction between a perverse sheaf on a closed subvariety $Y\subset X$ and its extension by zero to $X$; however, it is helpful to distinguish between $\cF\in\Perv_H(\cN_H,\bk)$ and $(a_H)_!\cF\in\Perv_H(\fh,\bk)$, where $a_H:\cN_H\hookrightarrow\fh$ is the inclusion.

\subsection{Cuspidal pairs}
\label{ss:cuspidal}

Throughout the paper, $G$ denotes a connected reductive complex algebraic group. Recall that $G$ has finitely many orbits in $\cN_G$, and that every simple object in $\Perv_G(\cN_G,\bk)$ is of the form $\IC(\cO,\cE)$ where $\cO\subset\cN_G$ is a $G$-orbit and $\cE$ is an irreducible $G$-equivariant $\bk$-local system on $\cO$. Such local systems on $\cO$ correspond to irreducible representations of the component group $A_G(x):=G_x/G_x^\circ$ on $\bk$-vector spaces, where $x$ is any element of~$\cO$.  

Let $P \subset G$ be a parabolic subgroup, with unipotent radical $U_P$, and let $L \subset P$ be a Levi factor. Then one can identify $L$ with $P/U_P$ through the natural morphism $L \hookrightarrow P \twoheadrightarrow P/U_P$, and thus define the diagram
\begin{equation}
\label{eqn:diagram-restriction-N}
\xymatrix@C=1.5cm{
\cN_L & \cN_P \ar[r]^-{i_{L \subset P}} \ar[l]_-{p_{L \subset P}} & \cN_G.
}
\end{equation}
Consider the functors
\begin{align*}
\Res^G_{L \subset P} & := (p_{L \subset P})_* \circ (i_{L \subset P})^! : \Db_G(\cN_G,\bk) \to \Db_L(\cN_L,\bk), \\ 
'\Res^G_{L \subset P} & := (p_{L \subset P})_! \circ (i_{L \subset P})^* : \Db_G(\cN_G,\bk) \to \Db_L(\cN_L,\bk).
\end{align*}
By \cite[Proposition 4.7]{ahr} and \cite[Proposition 3.1]{am}, these functors restrict to exact functors
\[
\Res^G_{L \subset P}, {}'\Res^G_{L \subset P} : \Perv_G(\cN_G,\bk) \to \Perv_{L}(\cN_{L},\bk).
\]
We also define the functor
\[
\Ind_{L \subset P}^G := \gamma_P^G \circ (i_{L \subset P})_! \circ (p_{L \subset P})^* : \Db_{L}(\cN_{L},\bk) \to \Db_G(\cN_G,\bk).
\]
(Here $\gamma_P^G$ is the left adjoint to the forgetful functor $\For^G_P$; see e.g.~\cite[\S B.10.1]{ahr} for a precise definition.)
Then $\Ind_{L \subset P}^G$ is left adjoint to $\Res^G_{L \subset P}$ and right adjoint to $'\Res^G_{L \subset P}$, and it also restricts to an exact functor (see \cite[Proposition 3.1]{am})
\[
\Ind_{L \subset P}^G : \Perv_{L}(\cN_{L},\bk) \to \Perv_G(\cN_G,\bk).
\]

\begin{prop}
\label{prop:conditions-cuspidal}
Let $\cF$ be a simple object in $\Perv_G(\cN_G,\bk)$. The following conditions are equivalent:
\begin{itemize}
\item[(1)] for any parabolic subgroup $P \subsetneq G$ and any Levi factor $L \subset P$ we have $\Res^G_{L \subset P} (\cF) =0$;
\item[$(1')$] for any parabolic subgroup $P \subsetneq G$ and any Levi factor $L \subset P$ we have $'\Res^G_{L \subset P} (\cF) =0$;
\item[(2)] for any parabolic subgroup $P \subsetneq G$ and Levi factor $L \subset P$, and for any object $\cG$ in $\Perv_{L}(\cN_{L},\bk)$, $\cF$ does not appear in the head of $\Ind_{L \subset P}^G(\cG)$;
\item[$(2')$] for any parabolic subgroup $P \subsetneq G$ and Levi factor $L \subset P$, and for any object $\cG$ in $\Perv_{L}(\cN_{L},\bk)$, $\cF$ does not appear in the socle of $\Ind_{L \subset P}^G(\cG)$.
\end{itemize}
\end{prop}

\begin{proof}
It is clear from adjunction that $(1)$ is equivalent to $(2)$ and that $(1')$ is equivalent to $(2')$. Hence we only have to prove that $(1)$ and $(1')$ are equivalent. Choose a maximal torus $T \subset L$. Let $P' \subset G$ be the parabolic subgroup opposite to $P$, i.e.~the parabolic subgroup whose Lie algebra has $T$-weights opposite to those of $P$. Then by \cite[Theorem 1]{braden}, the functors $\Res^G_{L \subset P}$ and $'\Res^G_{L \subset P'}$ are isomorphic, and similarly for $\Res^G_{L \subset P'}$ and $'\Res^G_{L \subset P}$. (Here we consider our functors as functors from $\Perv_G(\cN_G,\bk)$ to $\Perv_{L}(\cN_{L},\bk)$.) The equivalence of $(1)$ and $(1')$ follows.
\end{proof}

\begin{defn}
\label{def:cuspidal}
\begin{enumerate}
\item
A simple object $\cF$ in $\Perv_G(\cN_G,\bk)$ is called \emph{cuspidal} if it satisfies one of the conditions of Proposition \ref{prop:conditions-cuspidal}.
\item
A pair $(\cO,\mathcal{E})$, where $\cO \subset \cN_G$ is a $G$-orbit and $\mathcal{E}$ is an irreducible $G$-equivariant $\bk$-local system on $\cO$, is called \emph{cuspidal} if the perverse sheaf $\IC(\cO,\mathcal{E})$ is cuspidal.
\end{enumerate}
\end{defn}

\begin{rmk}
\label{rk:cuspidal}
\begin{enumerate}
\item
From the equivalence of conditions $(1)$ and $(1')$ in Proposition \ref{prop:conditions-cuspidal} we deduce in particular that a simple perverse sheaf is cuspidal if and only if its Verdier dual is cuspidal.
\item
One can easily check from the proof of Proposition \ref{prop:conditions-cuspidal} that in conditions $(2)$ and $(2')$ one can equivalently require $\cG$ to be simple.
\item
Mimicking terminology in other contexts, it is natural to call \emph{supercuspidal} a simple perverse sheaf which does not appear as a composition factor of any perverse sheaf of the form $\Ind_{L \subset P}^G(\cG)$ with $P \subsetneq G$.  (Note that in this condition again one can equivalently require $\cG$ to be simple.) Then of course supercuspidal implies cuspidal. If $\bk$ has characteristic zero, the decomposition theorem implies that any perverse sheaf $\Ind_{L \subset P}^G(\cG)$ with $\cG$ simple is semisimple. Hence, in this case any cuspidal simple perverse sheaf is supercuspidal. We will see in Remark~\ref{rmk:consequences-thm} that this is false when $\bk$ has positive characteristic.
\end{enumerate}
\end{rmk}

\subsection{Lusztig's original definition of cuspidal pairs}

The above definition of cuspidal perverse sheaves on the nilpotent cone follows~\cite{lusztig-fourier} (see also \cite{mirkovic}). 
Lusztig's original definition of cuspidal pairs~\cite{lusztig} was in the setting of $\Qlb$-sheaves (in the \'etale sense) on isolated classes in $G$. A direct translation of this definition into our setting would be to call a pair $(\cO,\mathcal{E})$ cuspidal if it satisfies the following condition for any parabolic subgroup $P \subsetneq G$, any Levi factor $L \subset P$, any $L$-orbit $\mathscr{C} \subset \cN_L$ and any $x \in \mathscr{C}$:
\begin{equation}
\label{eqn:cuspidal-lusztig}
\mathsf{H}_c^{\dim(\cO)-\dim(\mathscr{C})} \bigl(p_{L \subset P}^{-1}(x) \cap \cO, \mathcal{E} \bigr) = 0.
\end{equation}
In the $\bk=\Qlb$ case, Lusztig showed using the theory of weights that the two notions of cuspidal pair agree; see the proof of~\cite[23.2(b)]{lusztig-csdg5}. In the present more general setting, we have an implication in one direction only.

\begin{prop} \label{prop:sufficient}
Let $(\cO,\mathcal{E})$ be a pair consisting of a $G$-orbit $\cO \subset \cN_G$ and an irreducible $G$-equivariant $\bk$-local system $\mathcal{E}$ on $\cO$. Let $P \subset G$ be a parabolic subgroup and $L \subset P$ a Levi factor.
If $(\cO,\mathcal{E})$ satisfies condition \eqref{eqn:cuspidal-lusztig} for any $L$-orbit $\mathscr{C}\subset\cN_L$ and any $x\in\mathscr{C}$, then ${}'\Res^G_{L \subset P} \bigl( \IC(\cO,\mathcal{E}) \bigr)=0$.
\end{prop}

\begin{proof}
Since ${}'\Res^G_{L \subset P} \bigl( \IC(\cO,\mathcal{E}) \bigr)$ belongs to $\Perv_{L}(\cN_{L},\bk)$, it vanishes if and only if for any $L$-orbit $\mathscr{C} \subset \cN_L$ and any $x \in \mathscr{C}$ we have
\begin{equation}
\mathsf{H}^{-\dim(\mathscr{C})} \Bigl( {}'\Res^G_{L \subset P} \bigl( \IC(\cO,\mathcal{E}) \bigr)_x \Bigr)=0. 
\end{equation}
By definition of ${}'\Res^G_{L \subset P}$, this condition can be rewritten as
\begin{equation}
\mathsf{H}_c^{-\dim(\mathscr{C})} \bigl( p_{L \subset P}^{-1}(x) \cap \overline{\cO}, \IC(\cO,\mathcal{E}) \bigr)=0.
\end{equation}
Now the decomposition $\overline{\cO}=\cO \sqcup (\overline{\cO} \smallsetminus \cO)$ gives us an exact sequence
\begin{equation} \label{eqn:exact-seq}
\begin{split}
&\mathsf{H}_c^{-\dim(\mathscr{C})-1} \bigl( p_{L \subset P}^{-1}(x) \cap (\overline{\cO} \smallsetminus \cO), \IC(\cO,\mathcal{E}) \bigr) \\
&\quad\to \mathsf{H}_c^{\dim(\cO)-\dim(\mathscr{C})} \bigl( p_{L \subset P}^{-1}(x) \cap \cO, \mathcal{E} \bigr) \\
&\qquad\to \mathsf{H}_c^{-\dim(\mathscr{C})} \bigl( p_{L \subset P}^{-1}(x) \cap \overline{\cO}, \IC(\cO,\mathcal{E}) \bigr) \\
&\qquad\quad\to \mathsf{H}_c^{-\dim(\mathscr{C})} \bigl( p_{L \subset P}^{-1}(x) \cap (\overline{\cO} \smallsetminus \cO), \IC(\cO,\mathcal{E}) \bigr).
\end{split}
\end{equation}
We claim that the fourth term in this exact sequence is zero. 
To prove this, it suffices to prove that for any $G$-orbit $\cO' \subset \overline{\cO} \smallsetminus \cO$ we have
\begin{equation} \label{eqn:vanishing}
\mathsf{H}_c^{-\dim(\mathscr{C})} \bigl( p_{L \subset P}^{-1}(x) \cap \cO', \IC(\cO,\mathcal{E}) \bigr) = 0.
\end{equation}
However, by definition of the $\IC$ sheaf, $\IC(\cO,\mathcal{E})_{|\cO'}$ is concentrated in degrees $\leq -\dim(\cO')-1$. 
So~\eqref{eqn:vanishing} follows from the dimension estimate
\begin{equation}
\dim(p_{L \subset P}^{-1}(x) \cap \cO') \leq \frac{1}{2} \bigl( \dim(\cO')-\dim(\mathscr{C}) \bigr),
\end{equation}
proved by Lusztig in~\cite[Proposition 2.2]{lusztig}.
Consequently, we have a surjection
\begin{equation}
\mathsf{H}_c^{\dim(\cO)-\dim(\mathscr{C})} \bigl( p_{L \subset P}^{-1}(x) \cap \cO, \mathcal{E} \bigr) \twoheadrightarrow \mathsf{H}_c^{-\dim(\mathscr{C})} \bigl( p_{L \subset P}^{-1}(x) \cap \overline{\cO}, \IC(\cO,\mathcal{E}) \bigr),
\end{equation}
which implies the result.
\end{proof}

Lusztig classified the cuspidal pairs in the $\bk=\Qlb$ case in~\cite{lusztig}. As observed in~\cite[\S 2.2]{lusztig-cusp1}, when the group $G$ is defined over $\C$, the classification is unchanged if one considers nilpotent orbits rather than unipotent classes. Moreover, by general results, it makes no difference to consider $\Qlb$-sheaves for the strong topology rather than the \'etale topology, or to replace $\Qlb$ with another field $\bk$ of characteristic $0$ over which all the irreducible representations of the finite groups $A_G(x)$ for $x\in\cN_G$ are defined. Hence we can consider the classification of cuspidal pairs, as defined in Definition~\ref{def:cuspidal}, to be known when $\bk$ is such a field. As we will see, the classification when $\bk$ has characteristic $\ell>0$ is quite different.

\begin{rmk}
\label{rk:cuspidal-GL2}
Consider the case $G=\GL(2)$ when $\bk$ has characteristic $2$. We will see in Proposition \ref{prop:constant-cuspidal} below that the pair $(\cO_{(2)},\ubk)$ is cuspidal, where $\cO_{(2)}=\cN_G\setminus\{0\}$. For $B$ a Borel subgroup and $T$ a maximal torus we have
$p_{T \subset B}^{-1}(0) \cap \cO_{(2)} \cong \C^\times$.
In particular $\mathsf{H}_c^2(p_{T \subset B}^{-1}(0) \cap \cO_{(2)},\underline{\bk}) \neq 0$, so this cuspidal pair does not satisfy condition \eqref{eqn:cuspidal-lusztig}. Note that the first term of the exact sequence~\eqref{eqn:exact-seq} is nonzero in this case: in other words, the stalk of $\IC(\cO_{(2)},\ubk)$ at $0$ has nonzero $\mathsf{H}^{-1}$. This example shows that cuspidal perverse sheaves in positive characteristic need not be clean, in contrast with the case of characteristic $0$. Moreover, they need not satisfy parity vanishing.
\end{rmk}

\subsection{Transitivity}

Take $L\subset P\subset G$ as above. Suppose that $Q$ is a parabolic subgroup of $G$ containing $P$, with Levi factor $M\subset Q$ containing $L$. Then $L$ is also a Levi factor of the parabolic subgroup $P\cap M$ of $M$. We have the following well-known result expressing the transitivity of restriction and induction.

\begin{lem} \label{lem:transitivity}
With notation as above, we have isomorphisms of functors
\[
\begin{split}
\Res^G_{L\subset P}&\cong \Res^M_{L\subset P\cap M}\circ \Res^G_{M\subset Q}:\Db_G(\cN_G,\bk) \to \Db_L(\cN_L,\bk),\\
{}'\Res^G_{L\subset P}&\cong {}'\Res^M_{L\subset P\cap M}\circ {}'\Res^G_{M\subset Q}:\Db_G(\cN_G,\bk) \to \Db_L(\cN_L,\bk),\\
\Ind^G_{L\subset P}&\cong \Ind^G_{M\subset Q}\circ \Ind^M_{L\subset P\cap M}:\Db_L(\cN_L,\bk) \to \Db_G(\cN_G,\bk).
\end{split}
\]
\end{lem}

\begin{proof}
These are deduced from the basic base change and composition isomorphisms by standard arguments. In the present setting of equivariant derived categories, diagrams expressing these arguments may be found in~\cite[(7.1)]{ahr} (for $\Res$, and equivalently for ${}'\Res$) and~\cite[(7.8)]{ahr} (for $\Ind$). 
\end{proof}

\begin{cor} \label{cor:series}
Every simple object $\cF$ of $\Perv_G(\cN_G,\bk)$ appears in the head of $\Ind^G_{L\subset P}(\cG)$ for some parabolic subgroup $P$ of $G$, Levi factor $L\subset P$, and cuspidal simple perverse sheaf $\cG\in\Perv_L(\cN_L,\bk)$.
\end{cor}

\begin{proof}
We use induction on the semisimple rank of $G$. If $\cF$ is cuspidal, the conclusion holds with $L=P=G$ and $\cG=\cF$. In the base case of the induction, when $G$ is a torus, the unique simple object of $\Perv_G(\cN_G,\bk)$ is certainly cuspidal (since there are no proper parabolic subgroups). So we may assume that $G$ is not a torus and that $\cF$ is not cuspidal. By condition $(2)$ in Proposition~\ref{prop:conditions-cuspidal} and Remark \ref{rk:cuspidal}(2), there is some parabolic subgroup $Q \subsetneq G$, Levi factor $M\subset Q$, and simple object $\cH$ of $\Perv_{M}(\cN_{M},\bk)$ such that $\Ind_{M \subset Q}^G(\cH)$ surjects to $\cF$. By the induction hypothesis applied to $M$, there is some parabolic subgroup $P'\subset M$, Levi factor $L\subset P'$, and cuspidal simple perverse sheaf $\cG\in\Perv_{L}(\cN_{L},\bk)$ such that $\Ind_{L\subset P'}^M(\cG)$ surjects to $\cH$. Let $P=P'U_Q$ where $U_Q$ is the unipotent radical of $Q$; then $P$ is a parabolic subgroup of $G$ with Levi factor $L$ such that $P'=P\cap M$. Using the exactness of $\Ind_{M\subset Q}^G$ and Lemma~\ref{lem:transitivity}, we conclude that $\Ind_{L\subset P}^G(\cG)$ surjects to $\cF$ as required.   
\end{proof}

\begin{rmk}
A crucial point in Lusztig's proof of the generalized Springer correspondence in the $\bk=\Qlb$ case is that the pair $(L,\cG)$ in Corollary~\ref{cor:series} is uniquely determined up to $G$-conjugacy by $\cF$; thus, the simple objects of $\Perv_G(\cN_G,\bk)$ are partitioned into `induction series' indexed by the cuspidal pairs of Levi subgroups of $G$. In Theorem~\ref{thm:main} below, we will show that this remains true in the modular setting for $G=\GL(n)$.
\end{rmk}

\subsection{Fourier--Sato transform}
Let us fix a non-degenerate $G$-invariant symmetric bilinear form on the Lie algebra $\fg$ of $G$, and use it to identify $\fg$ and its dual. As in~\cite[\S 2.7]{ahjr}, we define the Fourier--Sato transform $\bT_{\fg}=\check q_!\circ q^* [\dim(\fg)]$, where $q,\check q:Q\to\fg$ are the first and second projections from a certain $G$-stable closed subset $Q\subset\fg\times\fg$. We regard this functor as an autoequivalence
\[
\bT_{\fg} : \Perv_G^{\mathrm{con}}(\fg,\bk) \simto \Perv_G^{\mathrm{con}}(\fg,\bk),
\]
where $\Perv_G^{\mathrm{con}}(\fg,\bk)$ is the full subcategory of $\Perv_G(\fg,\bk)$ consisting of conic objects. Notice that, for any $\cF\in\Perv_G(\cN_G,\bk)$, $(a_G)_!\cF$ belongs to $\Perv_G^{\mathrm{con}}(\fg,\bk)$, where $a_G:\cN_G \hookrightarrow \fg$ is the inclusion.

Let $P \subset G$ be a parabolic subgroup, and let $L \subset P$ be a Levi factor. Let $\fp$ and $\fl$ be the Lie algebras of $P$ and $L$, respectively. Then, instead of diagram \eqref{eqn:diagram-restriction-N}, one can consider the diagram
\begin{equation}
\label{eqn:diagram-restriction-g}
\xymatrix@C=1.5cm{
\fl & \fp \ar[r]^-{j_{L \subset P}} \ar[l]_-{q_{L \subset P}} & \fg,
}
\end{equation}
and set
\[
\begin{split}
{}'\underline{\Res}^G_{L \subset P} &:= (q_{L \subset P})_! \circ (j_{L \subset P})^* : \Db_G(\fg,\bk) \to \Db_L(\fl,\bk),\\
\underline{\Ind}_{L \subset P}^G &:= \gamma_P^G \circ (j_{L \subset P})_! \circ (q_{L \subset P})^* : \Db_{L}(\fl,\bk) \to \Db_G(\fg,\bk).
\end{split}
\]
Using the base change theorem one can easily construct isomorphisms of functors
\begin{equation}
\label{eqn:restrictions}
\begin{split}
{}'\underline{\Res}^G_{L\subset P} \circ (a_G)_! &\cong (a_L)_! \circ {}'\Res^G_{L\subset P},\\
\underline{\Ind}^G_{L\subset P} \circ (a_L)_! &\cong (a_G)_! \circ \Ind^G_{L\subset P}.
\end{split}
\end{equation}

The identification $\fg^* \cong \fg$ considered above induces an identification of subspaces $(\fg/\fu_P)^* \cong \fp$ (where $\fu_P$ is the nilpotent radical of $\fp$), and of their quotients $\fl^* \cong \fl$. In particular, one can then consider the Fourier--Sato transform $\bT_{\fl}$. The following lemma is easily checked, cf.~\cite[Lemma 4.2]{mirkovic}.

\begin{lem}
\label{lem:restriction-Fourier}
There are isomorphisms of functors
\[
{}'\underline{\Res}^G_{L \subset P} \circ \bT_{\fg} \cong \bT_{\fl} \circ {}'\underline{\Res}^G_{L \subset P}, \qquad \underline{\Ind}^G_{L \subset P} \circ \bT_{\fl} \cong \bT_{\fg} \circ \underline{\Ind}^G_{L \subset P}.
\]
\end{lem}

\begin{cor}
\label{cor:restriction-Fourier}
There are isomorphisms of functors
\[
{}'\underline{\Res}^G_{L \subset P} \circ \bT_{\fg} \circ (a_G)_! \cong \bT_{\fl} \circ (a_L)_! \circ {}'\Res^G_{L \subset P}, \qquad \underline{\Ind}^G_{L \subset P} \circ \bT_{\fl} \circ (a_L)_! \cong \bT_{\fg} \circ (a_G)_!\circ \Ind^G_{L \subset P}.
\]
\end{cor}

\begin{proof}
Combine Lemma~\ref{lem:restriction-Fourier} with~\eqref{eqn:restrictions}.
\end{proof}

The following result was shown by Lusztig for $\bk=\Qlb$~\cite{lusztig-fourier}. The argument below, which works for general $\bk$, is due to Mirkovi{\'c}~\cite{mirkovic}.  

\begin{prop}
\label{prop:Fourier-cuspidal}
Assume $G$ is semisimple. If $\cF$ is a cuspidal simple object of $\Perv_G(\cN_G,\bk)$, then $\bT_\fg ((a_G)_! \cF)$ is supported on $\cN_G$. If $\cF'$ is the unique simple object of $\Perv_G(\cN_G,\bk)$ such that $\bT_\fg ((a_G)_! \cF) = (a_G)_! \cF'$, then $\cF'$ is cuspidal.
\end{prop}

\begin{proof}
We observe that for any parabolic $P \subsetneq G$ and Levi factor $L \subset P$ we have
\[
{}'\underline{\Res}^G_{L \subset P} \bigl(\bT_\fg ((a_G)_! \cF)\bigr) \overset{{\rm Cor.~\ref{cor:restriction-Fourier}}}{\cong} \bT_\fl \bigl((a_L)_! {}'\Res^G_{L \subset P} (\cF)\bigr),
\]
which is $0$ by assumption. Then the first claim follows from \cite[Lemma 4.4]{mirkovic}. The second claim also follows from this observation and isomorphism \eqref{eqn:restrictions}.
\end{proof}

\begin{cor}
\label{cor:Fourier-cuspidal}
Drop the assumption that $G$ is semisimple, and let $\fz_G$ denote the centre of $\fg$. For any cuspidal pair $(\cO,\cE)$, there is a unique cuspidal pair $(\cO',\cE')$ such that 
\[ \bT_{\fg}((a_G)_!\IC(\cO,\cE))\cong\IC(\cO'+\fz_G,\cE'\boxtimes\ubk_{\fz_G}). \] 
\end{cor}

\begin{proof}
The semisimple group $G/Z(G)^\circ$ has Lie algebra $\fg/\fz_G$, which we can identify with the orthogonal complement of $\fz_G$ for our bilinear form on $\fg$ (i.e., the derived subalgebra of $\fg$).  
With this identification, the nilpotent cone and the set of cuspidal pairs for $G/Z(G)^\circ$ are the same as those for $G$. By Proposition~\ref{prop:Fourier-cuspidal}, for any cuspidal pair $(\cO,\cE)$ there is a unique cuspidal pair $(\cO',\cE')$ such that
\[
\bT_{\fg/\fz_G}((a_{G/Z(G)^\circ})_!\IC(\cO,\cE))\cong(a_{G/Z(G)^\circ})_!\IC(\cO',\cE').
\]
Applying~\cite[(2.14)]{ahjr} to the inclusion of $\fg/\fz_G$ in $\fg$, we deduce the isomorphism in the statement.
\end{proof}

\begin{rmk}
\label{rmk:Fourier-cuspidal}
In the $\bk=\Qlb$ case, Lusztig deduced from his classification of cuspidal pairs that one always has $(\cO',\cE')=(\cO,\cE)$ (see \cite[Theorem 5(b)]{lusztig-fourier}). For general $\bk$, the classification of cuspidal pairs is different, and Lusztig's argument applies only in some cases. However, we do not yet know an example where $(\cO',\cE')\neq(\cO,\cE)$. This issue does not arise when $G=\GL(n)$, because we will see in Theorem~\ref{thm:cuspidals-GL-new} that $\GL(n)$ has at most one cuspidal pair.
\end{rmk}

\subsection{Equivalent definitions of induction}

Let $L\subset P\subset G$ be as above. It is sometimes more convenient to use alternative descriptions of the induction functors $\Ind_{L\subset P}^G$ and $\underline{\Ind}^G_{L \subset P}$. In the setting of~\eqref{eqn:diagram-restriction-N} and~\eqref{eqn:diagram-restriction-g}, we factor the inclusions $i_{L\subset P}:\cN_P\hookrightarrow\cN_G$ and $j_{L\subset P}:\fp\hookrightarrow\fg$ as the compositions
\[
\xymatrix@R=2ex@C=1.5cm{
\cN_P \ar@{^{(}->}[r]^-{k_{L \subset P}} & G \times^P \cN_P \ar[r]^-{\mu_{L \subset P}} & \cN_G
\\
\fp \ar@{^{(}->}[r]^-{l_{L \subset P}} & G \times^P \fp \ar[r]^-{\pi_{L \subset P}} & \fg
}
\]
where $\mu_{L\subset P}$ and $\pi_{L\subset P}$ are the morphisms induced by the adjoint action of $G$ on $\fg$.

\begin{lem} \label{lem:other-ind}
There are isomorphisms of functors
\[
\begin{split}
\Ind_{L\subset P}^G&\cong (\mu_{L\subset P})_!\circ\mathsf{Ind}_P^G\circ(p_{L\subset P})^*[2\dim(G/P)],
\\
\underline{\Ind}_{L\subset P}^G&\cong (\pi_{L\subset P})_!\circ\mathsf{Ind}_P^G\circ(q_{L\subset P})^*[2\dim(G/P)],
\end{split}
\]
where $\mathsf{Ind}_P^G:\Db_P(X,\bk)\to\Db_G(G\times^P X,\bk)$ is the induction equivalence of~\cite[\S2.6.3]{bl}, for $X=\cN_P$ or $\fp$.
\end{lem}

\begin{proof}
The first isomorphism follows from:
\[
\begin{split}
\gamma_P^G \circ (i_{L \subset P})_! &\cong \gamma_P^G \circ (\mu_{L \subset P})_! \circ (k_{L \subset P})_! \\
&\cong (\mu_{L \subset P})_! \circ \gamma_P^G \circ (k_{L \subset P})_! \quad\text{(by~\cite[\S B.10.1]{ahr})}\\
&\cong (\mu_{L\subset P})_! \circ \mathsf{Ind}_P^G [2\dim(G/P)]\quad\text{(by~\cite[\S B.17]{ahr}).}
\end{split}
\]
The second is proved similarly.
\end{proof}

Recall that, if $\cO$ is an $L$-orbit in $\cN_L$, the $G$-orbit in $\cN_G$ \emph{induced by} $\cO$ is the unique dense $G$-orbit in $G\cdot(\overline{\cO}+\fu_P)$, where $\fu_P$ denotes the nilpotent radical of $\fp$.

\begin{cor} \label{cor:induced-orbit}
Let $\cF\in\Perv_L(\cN_L,\bk)$.
\begin{enumerate}
\item If the support of $\cF$ is contained in the closure of the $L$-orbit $\cO$, then the support of $\Ind_{L\subset P}^G(\cF)$ is contained in the closure of the induced $G$-orbit.
\item If $\cF$ is nonzero, then $\Ind_{L\subset P}^G(\cF)$ is nonzero.
\end{enumerate}
\end{cor}

\begin{proof}
We use the description of $\Ind_{L\subset P}^G$ given in Lemma~\ref{lem:other-ind}. For part (1), we see immediately that the support of $\Ind_{L\subset P}^G(\cF)$ is contained in $\mu_{L\subset P}\bigl((p_{L\subset P})^{-1}(\overline{\cO})\bigr)=G\cdot(\overline{\cO}+\fu_P)$, as required. Since $\Ind_{L\subset P}^G$ is exact, it suffices to prove part (2) in the case that $\cF=\IC(\cO,\cE)$ is a simple object. We have 
\[ (p_{L\subset P})^*\IC(\cO,\cE)\cong\IC(\cO+\fu_P,\cE\boxtimes\underline{\bk}_{\fu_P})[-\dim(\fu_P)]. \]
Denote by $\cE'$ the unique $G$-equivariant local system on $G\times^P(\cO+\fu_P)$ whose pull-back to $G\times(\cO+\fu_P)$ is 
$\underline{\bk}_G\boxtimes(\cE\boxtimes\underline{\bk}_{\fu_P})$. Then we have
\[
\mathsf{Ind}_P^G \bigl( \IC(\cO + \fu_P, \cE \boxtimes \underline{\bk}_{\fu_P}) \bigr) \cong \IC(G \times^P (\cO + \fu_P), \cE')[-\dim(G/P)].
\]
Let $\nu_{(L,\cO)}:G\times^P(\overline{\cO} + \fu_P)\to G\cdot(\overline{\cO}+\fu_P)$ be the restriction of $\mu_{L\subset P}$. Then we obtain an isomorphism
\begin{equation}
\Ind_{L \subset P}^G \bigl( \IC(\cO,\cE) \bigr) \cong (\nu_{(L,\cO)})_! \IC(G \times^P (\cO + \fu_P),\cE').
\end{equation}
Let $\cO^G\subset G\cdot(\overline{\cO}+\fu_P)$ be the induced orbit. Then $\dim\cO^G=\dim G\times^P(\overline{\cO} + \fu_P)$ by~\cite[Theorem 7.1.1]{cm}, so the restriction of $\nu_{(L,\cO)}$ to $\nu_{(L,\cO)}^{-1}(\cO^G)$ is a finite covering. Hence the restriction of $\Ind_{L \subset P}^G \bigl( \IC(\cO,\cE) \bigr)$ to $\cO^G$ is (up to shift) the push-forward of $\cE'$ through this finite covering, which is nonzero. 
\end{proof}

\subsection{Induction of some simple perverse sheaves}
\label{ss:induction-IC}

Let us fix some nilpotent orbit $\cO \subset \cN_L$ and some $L$-equivariant local system $\cE$ on $\cO$. Let $\fz_L$ be the center of $\fl$, and set
\[
\fz_L^\circ:= \{x \in \fz_L \mid G_x^\circ=L \}.
\]
In this subsection we recall, following arguments of Lusztig \cite{lusztig} adapted to our Lie algebra setting by Letellier \cite{letellier}, how one can describe the perverse sheaf $\uInd_{L \subset P}^G \bigl( \IC(\cO + \fz_L, \cE \boxtimes \underline{\bk}_{\fz_L} ) \bigr)$.

We set
\[
Y_{(L,\cO)} := G \cdot \bigl( \cO + \fz_L^\circ \bigr), \qquad X_{(L,\cO)} := G \cdot \bigl( \overline{\cO} + \fz_L + \fu_P \bigr).
\]
The subsets $Y_{(L,\cO)}$ are the strata in Lusztig's stratification of $\fg$. Recall that $Y_{(L,\cO)}$ is a locally closed smooth subvariety of $\fg$ and $X_{(L,\cO)}=\overline{Y_{(L,\cO)}}$ is a union of strata; see~\cite[Lemma 6.2, Proposition 6.5, Proposition 6.6]{lusztig-cusp2}. We also set
\[
\widetilde{Y}_{(L,\cO)} := G \times^L \bigl( \cO + \fz_L^\circ \bigr), \qquad \widetilde{X}_{(L,\cO)} := G \times^P \bigl( \overline{\cO} + \fz_L + \fu_P \bigr),
\]
and let $\xymap_{(L,\cO)} : \widetilde{Y}_{(L,\cO)} \to Y_{(L,\cO)}$ be the morphism induced by the adjoint $G$-action, and $\pi_{(L,\cO)} : \widetilde{X}_{(L,\cO)} \to X_{(L,\cO)}$ be the restriction of $\pi_{L\subset P}$, so that we obtain the following cartesian square:
\begin{equation}
\label{eqn:diagram-X-Y}
\vcenter{
\xymatrix{
\widetilde{Y}_{(L,\cO)} \ar@{^{(}->}[r] \ar[d]_-{\xymap_{(L,\cO)}} & \widetilde{X}_{(L,\cO)} \ar[d]^-{\pi_{(L,\cO)}} \\
Y_{(L,\cO)} \ar@{^{(}->}[r] & X_{(L,\cO)}.
}
}
\end{equation}
Here the upper horizontal map is induced by the natural morphism $G \times^L \fl \to G \times^P \fp$, and is an open embedding with image 
\[ \pi_{(L,\cO)}^{-1}(Y_{(L,\cO)}) = G \times^P (\cO + \fz_L^\circ + \fu_P), \]
by~\cite[Lemma 5.1.27]{letellier}.
We denote by $\widetilde{\cE}$ the unique $G$-equivariant local system on $\widetilde{Y}_{(L,\cO)}$ whose pull-back to $G \times \bigl( \cO + \fz_L^\circ \bigr)$ is $\underline{\bk}_G \boxtimes ( \cE \boxtimes \underline{\bk}_{\fz_L^\circ} )$.  Let
\[
N_G(L,\cO) := \{ n \in N_G(L) \mid n \cdot \cO = \cO \}.
\]

\begin{lem}[{\cite[proof of Lemma 5.1.28]{letellier}}]
\label{lem:Galois-covering}
The morphism $\xymap_{(L,\cO)}$ is a Galois covering with Galois group $N_G(L,\cO)/L$.  
\end{lem}

Consequently, given any irreducible $\bk[N_G(L,\cO)/L]$-module $E$, we can form the corresponding irreducible $G$-equivariant local system on $Y_{(L,\cO)}$:
\begin{equation}\label{eqn:Galois-covering}
\cL_E := ((\xymap_{(L,\cO)})_*\ubk \otimes E)^{N_G(L,\cO)/L}.
\end{equation}

The object $(\xymap_{(L,\cO)})_* \widetilde{\cE}$ is also a $G$-equivariant local system on $Y_{(L,\cO)}$, so that one can consider $\IC \bigl( Y_{(L,\cO)}, (\xymap_{(L,\cO)})_* \widetilde{\cE} \bigr)\in\Perv_G(\fg,\bk)$.

\begin{prop}
\label{prop:induction-IC}
There exists a canonical isomorphism
\[
\uInd_{L \subset P}^G \bigl( \IC(\cO + \fz_L, \cE \boxtimes \underline{\bk}_{\fz_L} ) \bigr) \cong \IC \bigl( Y_{(L,\cO)}, (\xymap_{(L,\cO)})_* \widetilde{\cE} \bigr).
\]
\end{prop}

\begin{proof}
Using the description of $\uInd_{L \subset P}^G$ given in Lemma~\ref{lem:other-ind}, and calculating as in the proof of Corollary~\ref{cor:induced-orbit}, we obtain a canonical isomorphism
\begin{equation}
\uInd_{L \subset P}^G \bigl( \IC(\cO + \fz_L, \cE \boxtimes \underline{\bk}_{\fz_L} ) \bigr) \cong (\pi_{(L,\cO)})_! \IC(G \times^P (\cO + \fz_L + \fu_P), \underline{\cE}),
\end{equation}
where $\underline{\cE}$ is the unique $G$-equivariant local system on $G \times^P \bigl( \cO + \fz_L + \fu_P \bigr)$ whose pull-back to $G \times \bigl( \cO + \fz_L + \fu_P \bigr)$ is $\underline{\bk}_G \boxtimes (\cE \boxtimes \underline{\bk}_{\fz_L} \boxtimes \underline{\bk}_{\fu_P})$.
So we need an isomorphism
\begin{equation}
\label{eqn:isom-IC}
(\pi_{(L,\cO)})_! \IC(G \times^P (\cO + \fz_L + \fu_P), \underline{\cE}) \cong \IC \bigl( Y_{(L,\cO)}, (\xymap_{(L,\cO)})_* \widetilde{\cE} \bigr).
\end{equation}
The isomorphism \eqref{eqn:isom-IC} is proved in \cite[Proposition 5.1.33]{letellier} for $\bk=\Qlb$. The proof works for arbitrary coefficients; for completeness we briefly recall the main steps. 

We first observe that the left-hand side of \eqref{eqn:isom-IC} is supported on $X_{(L,\cO)}=\overline{Y_{(L,\cO)}}$. To prove that it is isomorphic to the right-hand side we need to check that it satisfies the conditions defining intersection cohomology sheaves; see e.g.~\cite[\S4.1.1]{letellier}.

From the cartesian square~\eqref{eqn:diagram-X-Y} it is not difficult to deduce that the restriction of the left-hand side of \eqref{eqn:isom-IC} to $Y_{(L,\cO)}$ is isomorphic to the shifted local system $(\xymap_{(L,\cO)})_* \widetilde{\cE}[\dim(Y_{(L,\cO)})]$; see \cite[p.~76]{letellier}.
The other conditions that have to be checked concern the dimension of the support of the (ordinary) cohomology sheaves of $(\pi_{(L,\cO)})_! \IC(G \times^P (\cO + \fz_L + \fu_P), \underline{\cE})$. They follow from results on the dimension of the fibres of $\pi_{(L, \cO)}$ as in the case $\bk=\Qlb$; see \cite[p.~77]{letellier}.
\end{proof}

\begin{cor}
\label{cor:induction-IC}
Suppose that $(\cO,\cE)$ is a cuspidal pair for $L$. Then with $(\cO',\cE')$ defined as in Corollary~{\rm \ref{cor:Fourier-cuspidal}}, we have a canonical isomorphism
\[
\bT_{\fg}\bigl((a_G)_!\Ind_{L\subset P}^G(\IC(\cO,\cE))\bigr)\cong \IC \bigl( Y_{(L,\cO')}, (\xymap_{(L,\cO')})_* \widetilde{\cE'} \bigr).
\]
\end{cor}

\begin{proof}
Using Corollary~\ref{cor:restriction-Fourier} and Corollary~\ref{cor:Fourier-cuspidal} (applied to $L$), we find that the left-hand side is isomorphic to $\uInd_{L \subset P}^G ( \IC(\cO' + \fz_L, \cE' \boxtimes \underline{\bk}_{\fz_L} ))$. So the result follows from Proposition~\ref{prop:induction-IC}. 
\end{proof}

\begin{rmk}
In particular, it follows from Proposition \ref{prop:induction-IC} (repectively, Corollary~\ref{cor:induction-IC}) that $\uInd_{L \subset P}^G \bigl( \IC(\cO + \fz_L, \cE \boxtimes \underline{\bk}_{\fz_L} ) \bigr)$ (respectively, $\Ind_{L\subset P}^G(\IC(\cO,\cE))$) does not depend on $P$, up to canonical isomorphism.
\end{rmk}

We conclude this subsection by recalling some results about cuspidal pairs in characteristic~$0$ due to Lusztig. Note that \cite{lusztig} only treats the case of $\Qlb$-sheaves in the {\'e}tale setting; however one can easily check that the proofs adapt directly to our setting.

\begin{prop}[Lusztig]\label{prop:lusztig-cuspidal}
Assume that $\bk$ has characteristic $0$, and let $(\cO,\cE)$ be a cuspidal pair such that $\cE$ is absolutely irreducible.  Then:
\begin{enumerate}
\item We have $N_G(L,\cO) = N_G(L)$. Moreover, the isomorphism class of $\cE$ is preserved by the action of $N_G(L)/L$.\label{it:ngl}
\item There is a unique irreducible summand $\widebar{\cE}$ of the local system $(\xymap_{(L,\cO)})_*\widetilde{\cE}$ on $Y_{(L,\cO)}$ such that the cohomology sheaf $\cH^{-\dim(Y_{(L,\cO)})}\IC(Y_{(L,\cO)},\widebar{\cE})$ has nonzero restriction to the orbit in $\cN_G$ induced by $\cO$. The local system $\widebar{\cE}$ is absolutely irreducible, its multiplicity in $(\xymap_{(L,\cO)})_*\widetilde{\cE}$ is $1$, and moreover we have $(\xymap_{(L,\cO)})^*\widebar{\cE} \cong \widetilde{\cE}$. \label{it:loc-descend}
\item There is a $\bk$-algebra isomorphism $\End((\xymap_{(L,\cO)})_*\widetilde{\cE}) \cong \bk[N_G(L)/L]$ such that the resulting bijection
\begin{equation*}
\left\{\begin{array}{c}
\text{isomorphism classes of irreducible} \\
\text{summands of $(\xymap_{(L,\cO)})_*\widetilde{\cE}$}
\end{array}\right\}
\longleftrightarrow
\Irr(\bk[N_G(L)/L])
\end{equation*}
associates to $E \in \Irr(\bk[N_G(L)/L])$ the local system $\cL_E \otimes \widebar{\cE}$, where $\cL_E$ is as in~\eqref{eqn:Galois-covering}.\label{it:endo-ring}
\end{enumerate}
\end{prop}
\begin{proof}
Part~\eqref{it:ngl} is~\cite[Theorem~9.2(b)]{lusztig}. Part~\eqref{it:loc-descend} is all contained in~\cite[Theorem~9.2({c})]{lusztig} except for the statement that $(\xymap_{(L,\cO)})^*\widebar{\cE} \cong \widetilde{\cE}$, which follows (by standard Clifford theory) from the multiplicity-$1$ statement in view of part~\eqref{it:ngl}. The projection formula gives us an isomorphism $(\xymap_{(L,\cO)})_*\widetilde{\cE} \cong (\xymap_{(L,\cO)})_*\ubk \otimes \widebar{\cE}$, and part~\eqref{it:endo-ring} follows. 
\end{proof}

\begin{rmk}
The $\bk$-algebra isomorphism $\End((\xymap_{(L,\cO)})_*\widetilde{\cE}) \cong \bk[N_G(L)/L]$ defined in Proposition \ref{prop:lusztig-cuspidal}\eqref{it:endo-ring} is the same as that defined in~\cite[Theorem~9.2(d)]{lusztig}, by the uniqueness part of the latter statement.
\end{rmk}

\subsection{Modular reduction}
\label{ss:modular-reduction}

Let $\K$ be a finite extension of $\Ql$, $\O$ its ring of integers, and let $\F$ be its residue field. We will refer to such a triple $(\K,\O,\F)$ as an \emph{$\ell$-modular system.} Assume that for any $x \in \cN_G$ the irreducible representations of the finite group $A_G(x)$ are defined over $\K$. For $\E=\F$ or $\K$, we denote by $K_G(\cN_G,\E)$ the Grothendieck group of the category $\Db_G(\cN_G,\E)$. It is a free $\Z$-module with basis $\{[\cF]\}$ indexed by the isomorphism classes of simple objects $\cF$ in $\Perv_G(\cN_G,\E)$.

Recall~\cite[\S2.9]{juteau-aif} that there is a $\Z$-linear \emph{modular reduction} map
\[
d: K_G(\cN_G,\K) \to K_G(\cN_G,\F)
\]
which satisfies the following property. Let $\cF$ be an object in $\Perv_G(\cN_G,\K)$, and let $\cF_{\O}$ be a torsion-free object in $\Perv_G(\cN_G,\O)$ such that $\cF \cong \K \otimes_{\O} \cF_{\O}$;
then we have $d([\cF])=[\F \lotimes_{\O} \cF_{\O}]$.  By abuse of terminology, we say that a simple object $\cG$ in $\Perv_G(\cN_G,\F)$ `occurs in the modular reduction of $\cF$' if $[\cG]$ appears with non-zero multiplicity in $d([\cF])$. Note that if $\cF=\IC(\cO,\cE)$ for some $G$-orbit $\cO \subset \cN_G$ and some irreducible $G$-equivariant local system $\cE$ on $\cO$, then by our assumption on $\K$ there exists a $G$-equivariant $\O$-free local system $\cE_{\O}$ on $\cO$ such that $\cE \cong \K \otimes_{\O} \cE_{\O}$. In this situation we can take $\cF_{\O}=\IC(\cO,\cE_{\O})$; hence $d([\cF])=[\F \lotimes_\O \IC(\cO,\cE_\O)]$. In particular we deduce that if $\cE'$ is any composition factor of the $G$-equivariant local system $\F \lotimes_\O \cE_\O$, then $\IC(\cO,\cE')$ occurs in the modular reduction of $\cF$.

The following proposition is a crucial tool for identifying modular cuspidal pairs.

\begin{prop}
\label{prop:modular-reduction-cuspidal}
Let $\cG$ be a simple object in $\Perv_G(\cN_G,\F)$ that occurs in the modular reduction of a cuspidal simple object $\cF$ of $\Perv_G(\cN_G,\K)$. Then $\cG$ is cuspidal.
\end{prop}

\begin{proof}
Let $P \subsetneq G$ be a parabolic subgroup, and let $L \subset P$ be a Levi factor. Let $\cF_{\O}$ be a torsion-free object in $\Perv_G(\cN_G,\O)$ such that $\cF \cong \K \otimes_{\O} \cF_{\O}$; then $\cG$ is a composition factor of $\F \lotimes_{\O} \cF_{\O}$. 

Recall~\cite[\S2.1]{juteau-aif} that Verdier duality does not preserve $\Perv_G(\cN_G,\O)$.  Rather, it takes $\Perv_G(\cN_G,\O)$ to the heart of a different $t$-structure on $\Db_G(\cN_G,\O)$, denoted $\Perv^+_G(\cN_G,\O)$.  An object in $\Perv_G(\cN_G,\O)$ is torsion-free if and only if it also lies in $\Perv^+_G(\cN_G,\O)$.  

Consider the functors $\Res^G_{L \subset P}, {}'\Res^G_{L \subset P}: \Db_G(\cN_G, \O) \to \Db_L(\cN_L,\O)$, defined by the same formulas as in~\S\ref{ss:cuspidal}.  
By~\cite[Proposition~4.7]{ahr}, $\Res^G_{L \subset P}$ sends $\Perv_G(\cN_G,\O)$ to $\Perv_L(\cN_L,\O)$.  It follows that ${}'\Res^G_{L \subset P}$ sends $\Perv^+_G(\cN_G,\O)$ to $\Perv^+_L(\cN_L,\O)$.  On the other hand, the same argument as in the proof of~\cite[Proposition~4.7]{ahr} shows that ${}'\Res^G_{L \subset P}$ sends $\Perv_G(\cN_G,\O)$ to $\Perv_L(\cN_L,\O)$.  In particular, the object ${}'\Res^G_{L \subset P}(\cF_{\O})$ of $\Perv_L(\cN_L,\O)$ is torsion-free.

As explained in Remark~\ref{rmk:base-change} below, we have natural isomorphisms
\begin{equation} \label{eqn:tensor}
\begin{split}
\F \lotimes_{\O} {}'\Res^G_{L \subset P}(\cF_{\O}) &\cong {}'\Res^G_{L \subset P}(\F \lotimes_{\O} \cF_{\O}),\\
\K \otimes_{\O} {}'\Res^G_{L \subset P}(\cF_{\O}) &\cong {}'\Res^G_{L \subset P}(\K \otimes_{\O} \cF_{\O}) \cong {}'\Res^G_{L \subset P}(\cF).
\end{split}
\end{equation} 
Since $\cF$ is cuspidal, the second isomorphism in~\eqref{eqn:tensor} gives $\K \otimes_{\O} {}'\Res^G_{L \subset P}(\cF_{\O}) = 0$, which implies, by torsion-freeness, that ${}'\Res^G_{L \subset P}(\cF_{\O}) = 0$. Hence the first isomorphism in~\eqref{eqn:tensor} gives ${}'\Res^G_{L \subset P}(\F \lotimes_{\O} \cF_{\O})=0$, which implies that ${}'\Res^G_{L \subset P} (\cG) = 0$.
\end{proof}

\begin{rmk} \label{rmk:base-change}
In~\eqref{eqn:tensor} we used the fact that if $R$ and $S$ are noetherian commutative rings of finite global dimension, if $\phi : R \to S$ is a ring morphism and if $f : X \to Y$ is a morphism of locally compact topological spaces, then there exist natural isomorphisms of functors
\begin{equation} \label{eqn:ks}
(S \, \lotimes_R \, \cdot) \circ f_! \cong f_! \circ (S \, \lotimes_R \, \cdot), \qquad (S \, \lotimes_R \, \cdot) \circ f^* \cong f^* \circ (S \, \lotimes_R \, \cdot).
\end{equation}
The first isomorphism follows from~\cite[Proposition 2.6.6]{ks}, and the second from~\cite[Proposition 2.6.5]{ks}. These isomorphisms imply that any operation on the derived categories $\Db_G(\cN_G,\bk)$ that is obtained by composing functors of the form $f_!$ and $f^*$ commutes with modular reduction. 
\end{rmk}

\subsection{An example of modular reduction}

In this subsection, $\bk$ has characteristic $\ell>0$. Let $\K$ be a finite extension of $\Ql$ containing all the $n$-th roots of unity, and let $(\K,\O,\F)$ be the resulting $\ell$-modular system $(\K,\O,\F)$.

Consider the group $\mathrm{SL}(n)$, and its principal nilpotent orbit $\cO_{(n)}$. The component group $A_{\mathrm{SL}(n)}(x)$ of the centralizer of any element $x \in \cO_{(n)}$ is isomorphic to the group $\mu_n$ of complex $n$-th roots of unity; hence its group of characters with values in $\K$ is (non-canonically) isomorphic to the group of $n$-th roots of unity in $\K$.

\begin{lem}
\label{lem:primitive-cuspidal}
Let $\mathcal{E}$ be a rank-one $\mathrm{SL}(n)$-equivariant $\K$-local system on $\cO_{(n)}$ associated with a \emph{primitive} $n$-th root of unity in $\K$. Then $(\cO_{(n)},\mathcal{E})$ is a cuspidal pair for $\mathrm{SL}(n)$.
\end{lem}

\begin{proof}
This is one direction (in fact, the easy direction) of Lusztig's classification of cuspidal pairs for $\mathrm{SL}(n)$ in characteristic zero~\cite[(10.3.2)]{lusztig}. 
\end{proof}

Now consider $\GL(n)$. We have $\cN_{\GL(n)}=\cN_{\SL(n)}$, and the orbits are the same; however, for $x\in\cO_{(n)}$, the component group $A_{\mathrm{GL}(n)}(x)$ is trivial. So the only irreducible $\GL(n)$-equivariant $\bk$-local system on $\cO_{(n)}$ is the constant sheaf $\underline{\bk}$.

\begin{prop}
\label{prop:constant-cuspidal}
If $n$ is a power of $\ell$, then $(\cO_{(n)},\underline{\bk})$ is a cuspidal pair for $\GL(n)$.
\end{prop}

\begin{proof}
The conclusion does not depend on the choice of the field of characteristic $\ell$, so it is enough to prove it when $\bk=\F$. Let $\mathcal{E}$ be as in Lemma \ref{lem:primitive-cuspidal}. The modular reduction of $\mathcal{E}$ is $\underline{\F}$ (since $1$ is the only $n$-th root of unity in $\F$), so $\IC(\cO_{(n)},\underline{\F})$ occurs in the modular reduction of $\IC(\cO_{(n)},\mathcal{E})$. By Proposition~\ref{prop:modular-reduction-cuspidal}, we deduce that $(\cO_{(n)},\underline{\F})$ is a cuspidal pair for $\SL(n)$. However it is also a pair for $\GL(n)$, and the property of being cuspidal is the same whether we consider the pair for $\SL(n)$ or for $\GL(n)$ (see e.g.~condition $(1)$ in Proposition \ref{prop:conditions-cuspidal}).
\end{proof}

\section{Generalized Springer correspondence for $\GL(n)$}
\label{sec:mainthm}

For the remainder of this paper, we set $G=\GL(n)$ and assume that $\ell>0$.

\subsection{Combinatorics for the general linear group}
\label{ss:combinatorics}

In this subsection, we fix notation for partitions, representations, and nilpotent orbits.

\subsubsection{Compositions and partitions}

Let $\Comp$ denote the set of sequences of nonnegative integers with finitely many nonzero terms.  Elements of $\Comp$ are sometimes called \emph{compositions}.  For $\sa = (\sa_1, \sa_2, \ldots) \in \Comp$, let $\|\sa\| = \sum_{i=1}^\infty \sa_i$.  Given $\sa, \sfb \in \Comp$ and $k \in \N$, we can form the sum $\sa + \sfb$ and the product $k\sa$.

For $m \in \N$, let $\Part(m)$ denote the set of \emph{partitions} of $m$.  We identify $\Part(m)$ with the subset of $\Comp$ consisting of decreasing sequences $\lambda$ with $\|\lambda\| = m$.  For $\lambda \in \Part(m)$, $\mu \in \Part(m')$ and $k \in \N$, the sum $\lambda + \mu$ and the product $k\lambda$ are defined as above, via this identification.  For $\lambda \in \Part(m)$, let $\sm(\lambda) = (\sm_1(\lambda), \sm_2(\lambda), \ldots)$ be the composition in which $\sm_i(\lambda)$ is the multiplicity of $i$ in $\lambda$. We write $\lambda^\tr$ for the transpose partition, defined by the property that $\lambda_i^\tr-\lambda_{i+1}^\tr=\sm_i(\lambda)$ for all $i$.

Let $\Part_\ell(m) \subset \Part(m)$ be the set of \emph{$\ell$-regular partitions}, i.e., partitions in which $\sm_i(\lambda) < \ell$ for all $i$.  On the other hand, let $\Part(m,\ell) \subset \Part(m)$ be the set of partitions all of whose parts are powers of $\ell$: that is, $\sm_i(\lambda) = 0$ unless $i = \ell^j$ for some $j \ge 0$. For $\sa \in \Comp$, we define
\[
\uPart(\sa) = \prod_{i \ge 1} \Part(\sa_i)
\qquad\text{and}\qquad
\uPart_\ell(\sa) = \prod_{i \ge 1} \Part_\ell(\sa_i).
\]
Finally, let $[\sa] \in \uPart(\sa)$ be the element whose $i$-th component is the partition $(\sa_i) \in \Part(\sa_i)$.

\subsubsection{Representations of symmetric groups}

Next, consider the symmetric group $\fS_m$.  For $\lambda \in \Part(m)$, let $S^\lambda$ denote the Specht module for $\bk[\fS_m]$ corresponding to the partition $\lambda$.  If $\lambda \in \Part_\ell(m)$, we also let $D^\lambda$ denote the unique irreducible quotient of $S^\lambda$.  Recall that every irreducible $\bk[\fS_m]$-module arises in this way.  The partition $(m)$ is always $\ell$-regular, and $D^{(m)}$ is the trivial representation.

More generally, for $\sa \in \Comp$, let $\fS_\sa = \prod_{i \ge 1} \fS_{\sa_i}$.  For $\blambda = (\lambda^{(1)}, \lambda^{(2)}, \ldots)$ in $\uPart(\sa)$ (resp.~$\uPart_\ell(\sa)$), we can form the $\bk[\fS_\sa]$-module
\[
S^\blambda = S^{\lambda^{(1)}} \boxtimes S^{\lambda^{(2)}} \boxtimes \cdots
\qquad\text{resp.}\qquad
D^\blambda = D^{\lambda^{(1)}} \boxtimes D^{\lambda^{(2)}} \boxtimes \cdots.
\]
When $\blambda \in \uPart_\ell(\sa)$, $D^\blambda$ is the unique irreducible quotient of $S^\blambda$, and every irreducible $\bk[\fS_\sa]$-module arises in this way.  The trivial representation is $D^{[\sa]}$.

\subsubsection{Levi subgroups}

For each $\nu \in \Part(n)$, let $\bL_\nu$ denote the conjugacy class of Levi subgroups of $G$ that are isomorphic to
\[
\GL(\nu_1) \times \GL(\nu_2) \times \cdots.
\]
We will sometimes choose a representative Levi subgroup $L \in \bL_\nu$, and then consider the `relative Weyl group'
\[
W_\nu := N_G(L)/L \cong \fS_{\sm(\nu)}.
\]
The last isomorphism depends on the choice of $L$, but only up to an inner automorphism.  Thus, the set $\Irr(\bk[W_\nu])$ does not depend on the choice of $L$, and we have canonical bijections
\begin{equation}\label{eqn:irrw-identify}
\Irr(\bk[W_\nu]) = \Irr(\bk[\fS_{\sm(\nu)}]) = \uPart_\ell(\sm(\nu)).
\end{equation}

\subsubsection{Nilpotent orbits}

Let us label nilpotent orbits in $\cN_G$ by the set $\Part(n)$ of partitions of $n$ in the standard way, so that the partition $(n)$ corresponds to the principal nilpotent orbit, and $(1,\ldots,1)$ corresponds to the trivial orbit. If $\lambda \in \Part(n)$, we denote by $\cO_\lambda$ the corresponding orbit. It is well known that every $G$-equivariant local system on any $\cO_\la$ is constant, so we have a canonical bijection
\begin{equation}\label{eqn:pervn-identify}
\Irr(\Perv_G(\cN_G,\bk)) = \Part(n).
\end{equation}
For simplicity, we write $\IC_\la$ for $\IC(\cO_\la, \underline{\bk})$.

More generally, given $\nu \in \Part(n)$ and a representative $L \in \bL_\nu$, we can identify the set of nilpotent orbits in $\cN_L$ with $\uPart(\nu)$, and hence obtain a bijection
\begin{equation}
\label{eqn:bijection-orbits-N}
\Irr(\Perv_L(\cN_L,\bk)) = \uPart(\nu).
\end{equation}
For $\blambda \in \uPart(\nu)$, the corresponding orbit and perverse sheaf are denoted by $\cO_\blambda$ and $\IC_\blambda$, respectively.  In particular, $\cO_{[\nu]}$ is the principal orbit in $\cN_L$. Note that the bijection \eqref{eqn:bijection-orbits-N} is not canonical: it depends on a choice of isomorphism $L \cong \GL(\nu_1) \times \GL(\nu_2) \times \cdots$. However if $\blambda \in \uPart(\nu)$ satisfies the condition $\nu_i=\nu_j \Rightarrow \lambda^{(i)}=\lambda^{(j)}$ (e.g.~if $\blambda=[\nu]$) then $\cO_\blambda \subset L$ is canonically defined. This is the only case we will consider.

\subsection{Statements of the main results}
\label{subsect:statements-main}

Our first main result is a classification of the modular cuspidal pairs in $\GL(n)$.

\begin{thm}
\label{thm:cuspidals-GL-new}
The group $\GL(n)$ admits a cuspidal pair if and only if $n$ is a power of $\ell$. In that case, there exists a unique cuspidal pair, namely $(\cO_{(n)}, \ubk)$.
\end{thm}

\begin{rmk}
\label{rmk:consequences-thm}
Theorem \ref{thm:cuspidals-GL-new} provides the classification of cuspidal simple perverse sheaves in $\Perv_{\GL(n)}(\cN_{\GL(n)},\bk)$. The classification of \emph{supercuspidal} simple perverse sheaves (in the sense of Remark \ref{rk:cuspidal}(3)) is different and much simpler: in fact there is no supercuspidal simple perverse sheaf in $\Perv_{\GL(n)}(\cN_{\GL(n)},\bk)$ unless $n=1$. Even more is true: every simple object of $\Perv_{\GL(n)}(\cN_{\GL(n)},\bk)$ appears as a composition factor of the Springer sheaf $\Spr:=\Ind^{G}_{T\subset B}(\underline{\bk}_{\cN_T})$. (Here $T$ is a maximal torus and $B$ a Borel subgroup in $G=\GL(n)$.) Indeed it is sufficient to prove this property when $\bk=\Fl$. Now, the modular reduction (in the sense of \S\ref{ss:modular-reduction}) of the class of the Springer sheaf $[\Spr^{\Ql}]$ over $\Ql$ is the class of the Springer sheaf $[\Spr^{\Fl}]$ over $\Fl$. On the other hand every simple object $\IC(\cO_\la, \underline{\Ql})$ in $\Perv_{\GL(n)}(\cN_{\GL(n)},\Ql)$ is a direct summand of $\Spr^{\Ql}$; hence $\IC(\cO_\la, \underline{\Fl})$ occurs in the modular reduction of $\Spr^{\Ql}$.
\end{rmk}

Let $\fL = \{ \bL_\nu \mid \nu \in \Part(n,\ell) \}$.  An immediate consequence of Theorem~\ref{thm:cuspidals-GL-new} is that $\fL$ is precisely the set of conjugacy classes of Levi subgroups admitting a cuspidal pair, and that the unique cuspidal perverse sheaf for $L\in\bL_\nu$ is $\IC_{[\nu]}$.  In view of this observation, the following theorem can be regarded as a modular analogue of Lusztig's generalized Springer correspondence~\cite{lusztig}.

In this statement, for each $\nu \in \Part(n,\ell)$ we choose some $L \in \bL_\nu$ and some parabolic subgroup $P$ of $G$ having $L$ as a Levi factor, and set $\cI_\nu=\Ind_{L\subset P}^G(\IC_{[\nu]})$. We will see in the course of the proof that this object does not depend (up to isomorphism) on the choice of $L$ or of $P$.

\begin{thm}\label{thm:main}
For any $\lambda \in \Part(n)$, the perverse sheaf $\IC_\lambda$ appears in the head of $\cI_\nu$ for a unique $\nu\in\Part(n,\ell)$.
Moreover, for any $\nu\in\Part(n,\ell)$, we have a natural bijection
\[
\psi_\nu: \Irr(\bk[W_\nu]) \simto \left\{
\begin{array}{c}
\text{isomorphism classes of} \\
\text{simple quotients of $\cI_\nu$}
\end{array}
\right\}.
\]
In this way we obtain a bijective map
\[
\Psi = \bigsqcup_{\bL_\nu \in \fL} \psi_\nu: \bigsqcup_{\bL_\nu \in \fL} \Irr(\bk[W_\nu]) \simto \Irr(\Perv_G(\cN_G,\bk)).
\]
\end{thm}

The map $\psi_{\nu}$ will be defined intrinsically using the Fourier--Sato transform $\bT_\fg$ (see Lemma~\ref{lem:induction-cuspidal-new} below). The following result determines $\psi_{\nu}$ explicitly in terms of the combinatorial parameters, and thus effectively computes $\bT_\fg((a_G)_!\IC_\mu)$ for every simple $\IC_\mu\in\Perv_G(\cN_G,\bk)$ (see~\eqref{eqn:transform-rule}).

\begin{thm}\label{thm:explicit}
In the notation of Theorem~{\rm \ref{thm:main}}, the map $\psi_{\nu}$ is given explicitly by
\begin{equation}\label{eqn:genspring}
\psi_\nu(D^\blambda) = \IC_{\sum_{i \ge 0} \ell^i(\lambda^{(\ell^i)})^\tr}
\end{equation}
for $\blambda = (\lambda^{(1)}, \varnothing, \ldots, \varnothing, \lambda^{(\ell)}, \varnothing, \ldots, \varnothing, \lambda^{(\ell^2)}, \varnothing, \ldots) \in \uPart_\ell(\sm(\nu))$.
\end{thm}

See Figure~\ref{fig} for an example of the map $\Psi$.  (In that figure, elements of $\Irr(\bk[W_\nu])$ and $\Irr(\Perv_G(\cN_G,\bk))$ are represented by their combinatorial labels in diagrammatic form.)

\begin{figure}
\[
\begin{array}{|l|l|l|}
\hline
\begin{array}{@{}l@{}}
\nu \in \Part(6,2) \\ \bL_\nu \in \fL
\end{array}
 & \blambda \in \uPart_2(\sm(\nu)) & \psi_\nu(\blambda) \in \Part(6) \\
\hline
\hline
\begin{array}{@{}l@{}}
(1,1,1,1,1,1) \\ \GL(1)^6
\end{array}
 & \lambda^{(1)} \in \Part_2(6) & \\
 \cline{2-3}
 & \begin{tableau} :\c\c\c\c\c\c\\ \end{tableau}
 & \begin{tableau} :;\\:\c\\:\c\\:\c\\:\c\\:\c\\:\c\\:;\\ \end{tableau} \\
 \cline{2-3}
 & \begin{tableau} :\c\c\c\c\c\\:\c\\ \end{tableau}
 & \begin{tableau} :;\\:\c\c\\:\c\\:\c\\:\c\\:\c\\:;\\ \end{tableau} 
 \begin{tableau} :;\\:;\\:;\\:;\\:;\\:;\\:;\\:;\\ \end{tableau} \\
 \cline{2-3}
 & \begin{tableau} :\c\c\c\c\\:\c\c\\ \end{tableau}
 & \begin{tableau} :;\\:\c\c\\:\c\c\\:\c\\:\c\\:;\\ \end{tableau} 
 \begin{tableau} :;\\:;\\:;\\:;\\:;\\:;\\:;\\:;\\ \end{tableau} \\
 \cline{2-3}
 & \begin{tableau} :\c\c\c\\:\c\c\\:\c\\ \end{tableau}
 & \begin{tableau} :;\\:\c\c\c\\:\c\c\\:\c\\:;\\ \end{tableau}
 \begin{tableau} :;\\:;\\:;\\:;\\:;\\:;\\:;\\:;\\ \end{tableau} \\
\hline
\begin{array}{@{}l@{}}
(2,1,1,1,1) \\ \GL(2) \times \GL(1)^4
\end{array}
 & (\lambda^{(1)}, \lambda^{(2)}) \in \Part_2(4) \times \Part_2(1) & \\
 \cline{2-3}
 & \begin{tableau} :\c\c\c\c\\ \end{tableau} \ ,\ 
   \begin{tableau} :\c\\ \end{tableau}
 & \begin{tableau} :;\\:\c\c\c\\:\c\\:\c\\:\c\\:;\\ \end{tableau} 
 \begin{tableau} :;\\:;\\:;\\:;\\:;\\:;\\:;\\:;\\ \end{tableau} \\
 \cline{2-3}
 & \begin{tableau} :\c\c\c;\\:\c\\ \end{tableau} \ ,\ 
   \begin{tableau} :\c\\ \end{tableau}
 & \begin{tableau} :;\\:\c\c\c\c\\:\c\\:\c\\:;\\ \end{tableau} 
 \begin{tableau} :;\\:;\\:;\\:;\\:;\\:;\\:;\\:;\\ \end{tableau} \\
\hline
\begin{array}{@{}l@{}}
(2,2,1,1) \\ \GL(2)^2 \times \GL(1)^2
\end{array}
 & (\lambda^{(1)}, \lambda^{(2)}) \in \Part_2(2) \times \Part_2(2) & \\
 \cline{2-3}
 & \begin{tableau} :\c\c\\ \end{tableau} \ ,\ 
   \begin{tableau} :\c\c\\ \end{tableau}
 & \begin{tableau} :;\\:\c\c\c\\:\c\c\c\\:;\\ \end{tableau} 
 \begin{tableau} :;\\:;\\:;\\:;\\:;\\:;\\:;\\:;\\ \end{tableau} \\
\hline
\begin{array}{@{}l@{}}
(2,2,2) \\ \GL(2)^3
\end{array}
 & \lambda^{(2)} \in \Part_2(3) & \\
 \cline{2-3}
 & \begin{tableau} :\c\c\c\\ \end{tableau}
 & \begin{tableau} :;\\:\c\c\\:\c\c\\:\c\c\\:;\\ \end{tableau} 
 \begin{tableau} :;\\:;\\:;\\:;\\:;\\:;\\:;\\:;\\ \end{tableau} \\
 \cline{2-3}
 & \begin{tableau} :\c\c\\:\c\\ \end{tableau}
 & \begin{tableau} :;\\:\c\c\c\c\\:\c\c\\:;\\ \end{tableau}
 \begin{tableau} :;\\:;\\:;\\:;\\:;\\:;\\:;\\:;\\ \end{tableau} \\
\hline
\begin{array}{@{}l@{}}
(4,1,1) \\ \GL(4) \times \GL(1)^2
\end{array}
 & (\lambda^{(1)},\lambda^{(4)}) \in \Part_2(2) \times \Part_2(1) & \\
 \cline{2-3}
 & \begin{tableau} :\c\c\\ \end{tableau} \ ,\ 
   \begin{tableau} :\c\\ \end{tableau}
 & \begin{tableau} :;\\:\c\c\c\c\c\\:\c\\:;\\ \end{tableau}
 \begin{tableau} :;\\:;\\:;\\:;\\:;\\:;\\:;\\:;\\ \end{tableau} \\
\hline
\begin{array}{@{}l@{}}
(4,2) \\ \GL(4) \times \GL(2)
\end{array}
 & (\lambda^{(2)},\lambda^{(4)}) \in \Part_2(1) \times \Part_2(1) & \\
 \cline{2-3}
 & \begin{tableau} :\c\\ \end{tableau} \ ,\ 
   \begin{tableau} :\c\\ \end{tableau}
 & \begin{tableau} :;\\:\c\c\c\c\c\c\\:;\\ \end{tableau}
 \begin{tableau} :;\\:;\\:;\\:;\\:;\\:;\\:;\\:;\\ \end{tableau} \\
\hline
\end{array}
\]
\caption{The bijection $\Psi$ for $n = 6$ and $\ell = 2$} \label{fig}
\end{figure}

\subsection{Proofs of Theorems~\ref{thm:cuspidals-GL-new} and~\ref{thm:main}}
\label{subsect:induction}

These two theorems will be proved simultaneously by induction on $n$, the base case $n=1$ being trivial.

Given $\nu \in \Part(n)$, recall that we have chosen a representative $L \in \bL_\nu$, and a parabolic subgroup $P \subset G$ containing $L$.  
We will use the results of \S\ref{ss:induction-IC} for $L$, $P$, the $L$-orbit $\cO_{[\nu]}$, and the constant local system $\underline{\bk}$ on $\cO_{[\nu]}$. We simplify the notation by setting $Y_{[\nu]}:=Y_{(L,\cO_{[\nu]})}$, $X_{[\nu]}:=X_{(L,\cO_{[\nu]})}$, $\widetilde{Y}_{[\nu]}:=\widetilde{Y}_{(L,\cO_{[\nu]})}$, $\widetilde{X}_{[\nu]}:=\widetilde{X}_{(L,\cO_{[\nu]})}$, $\xymap_{[\nu]}:=\xymap_{(L,\cO_{[\nu]})}$, $\pi_{[\nu]}:=\pi_{(L,\cO_{[\nu]})}$. As a special case of Lemma~\ref{lem:Galois-covering} we have:

\begin{lem} \label{lem:galois}
The morphism $\xymap_{[\nu]}$ is a Galois covering with Galois group $W_\nu$.
\end{lem}

\begin{rmk}
\label{rmk:GL}
The varieties $Y_{[\nu]}$, $X_{[\nu]}$, $\widetilde{Y}_{[\nu]}$, $\widetilde{X}_{[\nu]}$ can be described very concretely in our case where $\fg=\fgl(n)$. The subset $Y_{[\nu]}\subset\fg$ consists of matrices whose generalized eigenspaces have dimensions $\nu_1,\nu_2,\ldots$ and which have a single Jordan block on each generalized eigenspace. Its closure $X_{[\nu]}$ consists of matrices for which the multiset of dimensions of generalized eigenspaces can be obtained from the multiset $\nu_1,\nu_2,\ldots$ by adding together some subsets. The variety $\widetilde{Y}_{[\nu]}$ can be identified with the variety of pairs $((U_i),x)$, where $(U_i)$ is an ordered tuple of subspaces of $\C^n$ such that $\C^n=\bigoplus_i U_i$ and $\dim U_i=\nu_i$, and $x\in Y_{[\nu]}$ has the subspaces $U_i$ as its generalized eigenspaces, in such a way that $\xymap_{[\nu]}((U_i),x)=x$.
The parabolic subgroup $P$ is the stabilizer of a partial flag in $\C^n$ with subspaces of dimensions $\nu_{\sigma(1)},\nu_{\sigma(1)}+\nu_{\sigma(2)},\nu_{\sigma(1)}+\nu_{\sigma(2)}+\nu_{\sigma(3)},\ldots$ where $\sigma$ is some permutation of the indices. Then $\widetilde{X}_{[\nu]}$ can be identified with the variety of pairs $((V_i),x)$, where $(V_i)$ is a partial flag with subspaces of these dimensions and $x\in X_{[\nu]}$ stabilizes each $V_i$ and has a single eigenvalue on each $V_i/V_{i-1}$, in such a way that $\pi_{[\nu]}((V_i),x)=x$.
\end{rmk}

\begin{lem}\label{lem:Fourier-cuspidal-GL}
Let $\nu \in \Part(n,\ell)$, and assume that $\nu \ne (n)$.  Then $\IC_{[\nu]}$ is the unique cuspidal perverse sheaf in $\Perv_{L}(\cN_{L},\bk)$.  Moreover, we have
\[
\bT_{\fl}\bigl((a_L)_!\IC_{[\nu]}\bigr) \cong \IC(\cO_{[\nu]} + \fz_L, \ubk).
\]
\end{lem}
\begin{proof}
The fact that $\IC_{[\nu]}$ is the unique cuspidal perverse sheaf for $L$ follows from Theorem~\ref{thm:cuspidals-GL-new}, which is known for each $\GL(\nu_i)$ by the induction hypothesis.
The claim about Fourier transform then follows from Corollary~\ref{cor:Fourier-cuspidal}.
\end{proof}

The following result, which generalizes facts shown by Lusztig in the $\bk=\Qlb$ case, is our main technical step. 
The $\nu=(1,1,\ldots,1)$ case of this statement is a particular case of~\cite[Proposition 5.4]{ahjr}.

\begin{lem}\label{lem:induction-cuspidal-new}
Let $\nu \in \Part(n,\ell)$, and assume that $\nu \ne (n)$.  Then the induced perverse sheaf
$\cI_\nu = \Ind_{L \subset P}^{G}(\IC_{[\nu]})$
enjoys the following properties:
\begin{enumerate}
\item Up to isomorphism, $\cI_\nu$ is independent of the choice of $P$ or $L$.\label{it:induction-indep}
\item There is a canonical isomorphism $\End(\cI_\nu) \cong \bk[W_\nu]$.\label{it:induction-endo}
\item Each indecomposable summand of $\cI_\nu$ has a simple head.\label{it:induction-summ}
\item There is a bijection \label{it:induction-head}
\[
\psi_\nu: \Irr(\bk[W_\nu]) \simto \left\{
\begin{array}{c}
\text{isomorphism classes of} \\
\text{simple quotients of $\cI_\nu$}
\end{array}
\right\}
\]
defined uniquely by the property that, for any $E\in \Irr(\bk[W_\nu])$,
\[
\bT_\fg\bigl((a_G)_!\psi_\nu(E)\bigr)\ \cong\ \IC(Y_{[\nu]}, \cE)
\]
where $\cE$ is the irreducible $G$-equivariant local system on $Y_{[\nu]}$ corresponding to $E$ as in~\eqref{eqn:Galois-covering}. In particular, $\psi_\nu$ is independent of the choice of $L$.
\end{enumerate}
\end{lem}

\begin{proof}
To prove these assertions, we will translate them into equivalent statements about the Fourier--Sato transform $\bT_\fg((a_G)_!\cI_\nu)$.  By Corollary~\ref{cor:induction-IC} and Lemma~\ref{lem:Fourier-cuspidal-GL}, we have a canonical isomorphism
\begin{equation} \label{eqn:anonymous}
\bT_\fg\bigl((a_G)_!\cI_\nu\bigr) \cong \IC(Y_{[\nu]}, \cL_{[\nu]}),
\end{equation}
where $\cL_{[\nu]}$ is the local system $(\xymap_{[\nu]})_*\underline{\bk}$ on $Y_{[\nu]}$. 
Since $Y_{[\nu]}$ and $(\xymap_{[\nu]})_* \underline{\bk}$ are independent of $P$ and $L$ (up to isomorphism), part~\eqref{it:induction-indep} follows. By Lemma~\ref{lem:galois}, we have $\End(\cL_{[\nu]}) \cong \bk[W_\nu]$.  We deduce part~\eqref{it:induction-endo} from the fact that the functor $\IC(Y_{[\nu]},\cdot )$ is fully faithful.

Since $\cL_{[\nu]}$ is the local system on $Y_{[\nu]}$ corresponding to the regular representation of $W_\nu$ as in~\eqref{eqn:Galois-covering}, the indecomposable summands of $\cL_{[\nu]}$ correspond to the indecomposable projective $\bk[W_\nu]$-modules, and they have simple heads corresponding to the irreducible $\bk[W_\nu]$-modules. Parts~\eqref{it:induction-summ} and~\eqref{it:induction-head} then follow from the fact that the functor $\IC(Y_{[\nu]},\cdot )$ preserves indecomposable summands and heads; see e.g.~\cite[Proposition 2.28]{juteau-aif}.
\end{proof}

We now need a purely combinatorial statement concerning the map which, according to the as-yet-unproved Theorem~\ref{thm:explicit}, computes the bijection $\psi_{\nu}$ introduced in Lemma~\ref{lem:induction-cuspidal-new}\eqref{it:induction-head}. We temporarily introduce a new notation for this map:
\[
\psico_\nu: \uPart_\ell(\sm(\nu)) \to \Part(n):\blambda\mapsto \sum_{i \ge 0} \ell^i(\lambda^{(\ell^i)})^\tr.
\]

\begin{lem} \label{lem:comb}
The following map is a bijection:
\[
\Psico = \bigsqcup_{\nu \in \Part(n,\ell)} \psico_\nu: \bigsqcup_{\nu \in \Part(n,\ell)} \uPart_\ell(\sm(\nu)) \to \Part(n).
\]
\end{lem}

\begin{proof}
For $\mu\in\Part(n)$, we can define integers $b_{i,j}(\mu)$ uniquely by
\[
\mu_j-\mu_{j+1}=\sum_{i\ge 0} b_{i,j}(\mu)\ell^i,\quad 0\leq b_{i,j}(\mu)<\ell.
\]
Then $\Psico(\blambda)=\mu$ if and only if $b_{i,j}(\mu)=\sm_j(\lambda^{(\ell^i)})$ for all $i,j$, which clearly holds for a unique $\blambda\in\uPart_\ell(\sm(\nu))$ for a unique $\nu\in\Part(n,\ell)$. 
\end{proof}

Recall that an example of $\Psico$ was given in Figure~\ref{fig}.  

We can now complete the proof of Theorems~\ref{thm:cuspidals-GL-new} and~\ref{thm:main}.
Suppose that $n$ is not a power of $\ell$, so that $(n)\notin\Part(n,\ell)$. Assembling the bijections $\psi_{\nu}$ we obtain a map
\[
\Psi = \bigsqcup_{\bL_\nu \in \fL} \psi_\nu: \bigsqcup_{\bL_\nu \in \fL} \Irr(\bk[W_\nu]) \longrightarrow \Irr(\Perv_G(\cN_G,\bk)).
\]
Now the images of the various $\psi_{\nu}$ are disjoint by Lemma~\ref{lem:induction-cuspidal-new}\eqref{it:induction-head}, because the various subsets $Y_{[\nu]}$ are disjoint. (This is clear from Remark~\ref{rmk:GL}; alternatively, it follows from the general result~\cite[\S6.1]{lusztig-cusp2}.)  
Hence $\Psi$ is injective. From Lemma~\ref{lem:comb} we know that the domain and codomain of $\Psi$ have the same cardinality, so $\Psi$ is a bijection. Thus every simple perverse sheaf in $\Perv_G(\cN_G,\bk)$ arises as a quotient of a sheaf induced from a proper Levi subgroup, meaning that there are no cuspidal perverse sheaves for $G$. 

Suppose that $n$ is a power of $\ell$. Assembling the bijections $\psi_{\nu}$ for $\nu\neq(n)$ we obtain a map
\[
\Psi' = \bigsqcup_{\substack{\bL_\nu \in \fL\\\bL_\nu\text{ proper}}} \psi_\nu: \bigsqcup_{\substack{\bL_\nu \in \fL\\\bL_\nu\text{ proper}}} \Irr(\bk[W_\nu]) \longrightarrow \Irr(\Perv_G(\cN_G,\bk))
\]
which is injective as before. Since $W_{(n)}$ is the trivial group, the domain here has cardinality one less than the codomain. On the other hand, we already know from Proposition~\ref{prop:constant-cuspidal} that $\IC_{(n)}$ is cuspidal, so it must be the unique element of the codomain not appearing in the image of $\Psi'$. Hence it is the unique cuspidal perverse sheaf for $G$, and we can complete $\Psi'$ to the desired bijection $\Psi$ by defining $\psi_{(n)}(\bk)=\IC_{(n)}$, where $\bk$ denotes the unique element of $\Irr(\bk[W_{(n)}])$. This completes the proof of Theorems~\ref{thm:cuspidals-GL-new} and~\ref{thm:main}. 
\qed

\bigskip

Note that the assumption $\nu \ne (n)$ in Lemma~\ref{lem:Fourier-cuspidal-GL} was needed solely in order to apply the induction hypothesis to each factor of $L$. So the same proof now applies to the $\nu=(n)$ case also (when $n$ is a power of $\ell$), showing that
\begin{equation}
\bT_\fg\bigl((a_G)_!\IC_{(n)}\bigr)\cong\IC(\cO_{(n)} + \fz_G,\underline{\bk})=\IC(Y_{(n)},\underline{\bk}).
\end{equation}
Thus the intrinsic definition of $\psi_{\nu}$ given in Lemma~\ref{lem:induction-cuspidal-new}\eqref{it:induction-head} is valid for $\nu=(n)$ also.

\subsection{Proof of Theorem~\ref{thm:explicit}}
\label{subsect:proofs}

Identifying the domain and codomain of $\Psi$ with their respective sets of combinatorial parameters, we can interpret $\Psi$ as a bijection
\[
\Psi = \bigsqcup_{\nu \in \Part(n,\ell)} \psi_\nu: \bigsqcup_{\nu \in \Part(n,\ell)} \uPart_\ell(\sm(\nu)) \to \Part(n).
\]
Theorem~\ref{thm:explicit} asserts that this bijection equals the combinatorially-defined bijection $\Psico$ of Lemma~\ref{lem:comb}. Equivalently, it says that for any $\nu\in\Part(n,\ell)$ and $\blambda\in\uPart_\ell(\sm(\nu))$, we have
\begin{equation} \label{eqn:transform-rule}
\bT_\fg\bigl((a_G)_!\IC_{\psico_\nu(\blambda)}\bigr)\cong \IC(Y_{[\nu]},\cD^{\blambda}),
\end{equation}
where $\cD^{\blambda}$ is the irreducible $G$-equivariant local system on $Y_{[\nu]}$ associated to $D^{\blambda}\in\Irr(\bk[W_\nu])$ as in~\eqref{eqn:Galois-covering}. Note that the $\nu=(1,1,\ldots,1)$ case of this statement is exactly the determination of the modular Springer correspondence given in~\cite[Theorem 6.4.1]{juteau}; in fact the general proof will use a similar idea. (In view of~\eqref{eqn:transform-rule}, Figure~\ref{fig} can be interpreted as recording the Fourier--Sato transforms of all simple perverse sheaves in $\Perv_{\GL(6)}(\cN_{\GL(6)}, \bk)$ where $\bk$ has characteristic~$2$.)

We will prove Theorem~\ref{thm:explicit} by induction on $n$, the base case $n=1$ being trivial. Since we know that $\Psi$ (interpreted as above) and $\Psico$ are bijections with the same domain and codomain, it suffices to prove that for each $\nu\in\Part(n,\ell)$ we have
\begin{equation}\label{eqn:main-ineq-new}
\psi_\nu(\blambda) \le \psico_\nu(\blambda) \qquad\text{for all $\blambda \in \uPart_\ell(\sm(\nu))$,}
\end{equation}
where $\le$ denotes the usual dominance partial order on partitions, corresponding to the closure order on nilpotent orbits. (However, our induction hypothesis will still have the equality $\psi_\nu=\psico_\nu$ rather than this inequality.) 

\begin{lem}\label{lem:induction-mli}
Assume that $n = m\ell^i$ for some $m>1$ and $i\geq 0$. Then~\eqref{eqn:main-ineq-new} holds for the partition $\nu = (\ell^i, \ell^i, \ldots, \ell^i)$.
\end{lem}

\begin{proof}
Since the bijection $\Psi$ does not depend on which field of characteristic $\ell$ we use, we can assume for this proof that $\bk=\F$ where $(\K,\O,\F)$ is an $\ell$-modular system as in the proof of Proposition~\ref{prop:constant-cuspidal}. We will deduce the result from the known description of the generalized Springer correspondence for the group $G' = \mathrm{SL}(n)$ over the field $\K$ of characteristic zero~\cite{lus-spalt}. To be precise,~\cite{lus-spalt} treats the unipotent variety in $\mathrm{SL}(n)$, for $\Qlb$-sheaves in the {\'e}tale setting; but the same arguments can be carried out in our setting.

Let $L'=G' \cap L \cong \mathrm{S}(\GL(\ell^i) \times \cdots \times \GL(\ell^i))$ be the Levi subgroup of $G'$ corresponding to $L$, and let $\fl'$ denote its Lie algebra.
Let $A_{G'}$ (resp.~$A_{L'}$) denote the component group of the centre of $G'$ (resp.~$L'$).  Then $A_{G'}$ (resp.~$A_{L'}$) is cyclic of order $n$ (resp.~$\ell^i$).  Fix a generator $z \in A_{G'}$, and let $\bar z$ denote its image in $A_{L'}$.
Define $Y'_{[\nu]}=Y_{(L',\cO_{[\nu]})}=Y_{[\nu]}\cap\fg'$ and $\widetilde{Y}'_{[\nu]}=\widetilde{Y}_{(L',\cO_{[\nu]})}$, which we can obviously identify with $\widetilde{Y}_{[\nu]}\cap(G\times^L\fl')$. Then $\xymap_{[\nu]}$ restricts to the Galois covering $\xymap_{[\nu]}'=\xymap_{(L',\cO_{[\nu]})}:\widetilde{Y}'_{[\nu]} \to Y'_{[\nu]}$ with Galois group $W_\nu \cong \fS_m$.

Fix a primitive $\ell^i$-th root of unity $\zeta$ in $\K$, and let $\cE_{[\nu]}^\zeta$ be the unique irreducible $L'$-equivariant $\K$-local system on $\cO_{[\nu]}$ on which $\bar z\in A_{L'}$ acts by $\zeta$. Then $(\cO_{[\nu]},\cE_{[\nu]}^\zeta)$ is a cuspidal pair for $L'$, by 
the classification of cuspidal pairs in characteristic zero (see~\cite[\S 10.3]{lusztig}).

We associate to $(\cO_{[\nu]},\cE_{[\nu]}^\zeta)$ a $G'$-equivariant local system $\widetilde{\cE}_{[\nu]}^\zeta$ on $\widetilde{Y}'_{[\nu]}$ as in \S\ref{ss:induction-IC}, and a $G'$-equivariant local system $\widebar{\cE}_{[\nu]}^\zeta$ on $Y'_{[\nu]}$ as in Proposition~\ref{prop:lusztig-cuspidal}\eqref{it:loc-descend}. For any $\lambda \in \Part(m)$, we have an irreducible $\K[\fS_m]$-module $S^\lambda_\K$ and a corresponding local system $\cS^\lambda_\K$ on $Y'_{[\nu]}$, defined as in~\eqref{eqn:Galois-covering}.  By Proposition~\ref{prop:lusztig-cuspidal}\eqref{it:endo-ring}, the irreducible summands of $\IC(Y'_{[\nu]}, (\xymap'_{[\nu]})_*\widetilde{\cE}_{[\nu]}^\zeta)$ are of the form $\IC(Y'_{[\nu]}, \cS^\lambda_\K \otimes \widebar{\cE}_{[\nu]}^\zeta)$.  For brevity, we set $\cS^{\lambda,\zeta}_\K = \cS^\lambda_\K \otimes \widebar{\cE}_{[\nu]}^\zeta$.

According to~\cite[Theorem 6.5(c)]{lusztig} and~\cite[Proposition~5.2]{lus-spalt}, we have
\begin{equation}
(a_{G'})^*\IC(Y'_{[\nu]}, \cS^{\lambda,\zeta}_\K) \cong \IC(\cO_{\ell^i\lambda},\cE_{\ell^i\lambda}^\zeta)[\dim(Y'_{[\nu]}) - \dim(\cO_{\ell^i\lambda})],
\end{equation}
where $\cE_{\ell^i\lambda}^\zeta$ is the unique irreducible $G'$-equivariant $\K$-local system on $\cO_{\ell^i\lambda}$ on which $z\in A_G$ acts by $\zeta$. On the other hand, by~\cite[\S3.7 and Theorem~3.8(c)]{evens-mirk}, we have
\begin{equation}
\IC(Y'_{[\nu]}, \cS^{\lambda,\zeta}_\K)\cong \bT_{\fg'}\bigl((a_{G'})_!(a_{G'})^*\IC(Y'_{[\nu]}, \cS^{\lambda^\tr,\zeta}_\K)[-\dim(Y'_{[\nu]})+\dim(\cO_{\ell^i\lambda^\tr})]\bigr).
\end{equation}
We conclude that
\begin{equation} \label{eqn:fundamental}
\IC(Y'_{[\nu]}, \cS^{\lambda,\zeta}_\K)\cong \bT_{\fg'}\bigl((a_{G'})_!\IC(\cO_{\ell^i\lambda^\tr},\cE_{\ell^i\lambda^\tr}^\zeta)\bigr).
\end{equation}

From now on we assume that $\lambda \in \Part_{\ell}(m)$.
We have already defined (after~\eqref{eqn:transform-rule}) the $\F$-local system $\cD^\lambda_\F$ on $Y_{[\nu]}$; let $(\cD^\lambda_\F)'$ denote its restriction to $Y'_{[\nu]}$, which is again the local system corresponding to the irreducible $\F[\fS_m]$-module $D^\lambda_\F$, but now for the restricted covering map $\xymap'_{[\nu]}$. By definition of $\psi_\nu$ we have
\begin{equation}
\IC(Y_{[\nu]},\cD^\lambda_\F)\cong\bT_\fg\bigl((a_{G})_!\IC(\cO_{\psi_\nu(\lambda)},\underline{\F})\bigr),
\end{equation}
which implies that
\begin{equation} \label{eqn:auxiliary}
\IC(Y'_{[\nu]},(\cD^\lambda_\F)')\cong\bT_{\fg'}\bigl((a_{G'})_!\IC(\cO_{\psi_\nu(\lambda)},\underline{\F})\bigr).
\end{equation}

Now we want to consider the `modular reduction' of~\eqref{eqn:fundamental}. We have not defined modular reduction for general perverse sheaves on $\fg'$, only for those in $\Perv_{G'}(\cN_{G'},\K)$. However, the left-hand side of~\eqref{eqn:fundamental} is of a form that we can treat merely by considering local systems with finite monodromy (effectively, representations of finite groups).  Choose torsion-free $\O$-forms $\cS^\lambda_\O$ and $\widebar{\cE}_{[\nu],\O}^\zeta$ of $\cS^\lambda_\K$ and $\widebar{\cE}_{[\nu]}^\zeta$, respectively, and let $\cS^{\lambda,\zeta}_\O = \cS^\lambda_\O \otimes \widebar{\cE}_{[\nu],\O}^\zeta$, a torsion-free $\O$-form of $\cS^{\lambda,\zeta}_\K$.  Since $\widebar{\cE}_{[\nu]}^\zeta$ has rank $1$ and order a power of $\ell$, we see that $\F \otimes_\O \widebar{\cE}_{[\nu],\O}^\zeta \cong \ubk$.  In particular,
\begin{equation}
\F \otimes_\O \cS^{\lambda,\zeta}_\O \cong \F \otimes_\O \cS^\lambda_\O.
\end{equation}
It follows that $(\cD^\lambda_\F)'$ is a constituent of $\F\otimes_{\O}\cS^{\lambda,\zeta}_\O$. Hence $\IC(Y'_{[\nu]},(\cD^\lambda_\F)')$ is a constituent of $\F\lotimes_{\O}\IC(Y'_{[\nu]},\cS^{\lambda,\zeta}_\O)$. Since Fourier--Sato transform is an equivalence and commutes with change of scalars by~\eqref{eqn:ks}, we conclude from~\eqref{eqn:fundamental} and~\eqref{eqn:auxiliary} that $\IC(\cO_{\psi_\nu(\lambda)},\underline{\F})$ occurs in the modular reduction of $\IC(\cO_{\ell^i\lambda^\tr},\cE_{\ell^i\lambda^\tr}^\zeta)$. In particular, $\cO_{\psi_\nu(\lambda)}$ is contained in the closure of $\cO_{\ell^i\lambda^\tr}=\cO_{\psico_\nu(\lambda)}$, which gives the desired inequality.
\end{proof}

The proof of Theorem~\ref{thm:explicit} is completed by:

\begin{lem} 
If $\nu\in\Part(n,\ell)$ is not of the form $(\ell^i,\ell^i,\ldots,\ell^i)$, then~\eqref{eqn:main-ineq-new} holds.
\end{lem}

\begin{proof}
Let $m_i = \sm_{\ell^i}(\nu)$, and let $\hat\nu$ be the partition whose parts are of the form $m_i\ell^i$, $i \ge 0$.  We can assume that the representatives $L \in \bL_\nu$ and $M \in \bL_{\hat\nu}$ are such that $L \subset M$, and such that the inclusion $L \hookrightarrow M$ is compatible with the direct product decompositions of those groups in the following way:
\[
\begin{array}{c@{}c@{}c@{}c@{}c@{}c@{}c@{}c@{}c}
L &{}={}&      \underbrace{\GL(1) \times \cdots \times \GL(1)}_{\text{$m_0$ copies}}
  &{}\times{}& \underbrace{\GL(\ell) \times \cdots \times \GL(\ell)}_{\text{$m_1$ copies}}
  &{}\times{}& \underbrace{\GL(\ell^2) \times \cdots \times \GL(\ell^2)}_{\text{$m_2$ copies}}
  &{}\times{}& \cdots \\
\downarrow && \downarrow && \downarrow && \downarrow \\
M &{}={}&      \GL(m_0)
  &{}\times{}& \GL(m_1\ell)
  &{}\times{}& \GL(m_2\ell^2)
  &{}\times{}& \cdots
\end{array}
\]
We can also assume that the corresponding parabolic subgroups $P \supset L$ and $Q \supset M$ are such that $P \subset Q$.  Let $\cI^M_\nu = \Ind_{L \subset P \cap M}^M \IC_{[\nu]}$. 

By assumption, $M$ is a proper subgroup of $G$, so we can apply the induction hypothesis to each factor of $M$. Notice that the relative Weyl group $N_M(L)/L$ is the same as $N_G(L)/L$, namely it equals $W_\nu=\fS_{m_0}\times\fS_{m_1}\times\fS_{m_2}\times\cdots$. By Lemma~\ref{lem:induction-cuspidal-new} applied to each factor of $M$, the simple quotients of $\cI^M_\nu$ are canonically indexed by $\uPart_\ell(\sm(\nu))$; let $\cF^{\blambda,M}$ denote the simple quotient corresponding to $\blambda\in\uPart_\ell(\sm(\nu))$. The induction hypothesis tells us that
\begin{equation}\label{eqn:im-simple}
\cF^{\blambda,M} \cong \IC_{(\lambda^{(1)})^\tr} \boxtimes \IC_{\ell \cdot (\lambda^{(\ell)})^\tr} \boxtimes \IC_{\ell^2 \cdot (\lambda^{(\ell^2)})^\tr} \boxtimes \cdots.
\end{equation}
By Corollary~\ref{cor:induced-orbit}(1), the support of $\Ind_{M \subset Q}^G \cF^{\blambda,M}$ is contained in the closure of the orbit in $\cN_G$ induced by the orbit $\cO_{(\lambda^{(1)})^\tr}\times\cO_{\ell\cdot(\lambda^{(\ell)})^\tr}\times\cO_{\ell^2\cdot(\lambda^{(\ell^2)})^\tr}\cdots$ in $\cN_M$. By \cite[Lemma~7.2.5]{cm}, that induced orbit is $\cO_{\psico_\nu(\blambda)}$. So to prove~\eqref{eqn:main-ineq-new}, it suffices to prove that
\begin{equation}
\label{eqn:ind-Fnu}
\Ind_{M \subset Q}^G \cF^{\blambda,M} \quad \text{surjects to} \quad \IC_{\psi_\nu(\blambda)}\qquad\text{for all $\blambda\in\uPart_\ell(\sm(\nu))$.}
\end{equation}

Let $G^\blambda$ denote the projective cover of $D^\blambda$ as a $\bk[W_\nu]$-module, and let $\cG^\blambda$ denote the corresponding local system on $Y_{[\nu]}$. As mentioned in the proof of Lemma~\ref{lem:induction-cuspidal-new}, the indecomposable direct summands of $\IC(Y_{[\nu]},\cL_{[\nu]})$ have the form $\IC(Y_{[\nu]},\cG^\blambda)$ for $\blambda\in\uPart_\ell(\sm(\nu))$, and the head of $\IC(Y_{[\nu]},\cG^\blambda)$ is $\IC(Y_{[\nu]},\cD^\blambda)$. Since $\IC(Y_{[\nu]},\cL_{[\nu]})\cong\bT_{\fg}((a_G)_!\cI_\nu)$ and $\bT_{\fg}$ is an equivalence, the indecomposable summands of $\cI_\nu$ are also indexed by $\uPart_\ell(\sm(\nu))$: if $\mathcal{Q}^\blambda$ denotes the indecomposable summand of $\cI_\nu$ such that $\bT_{\fg}((a_G)_!\mathcal{Q}^\blambda)\cong\IC(Y_{[\nu]},\cG^\blambda)$, then the head of $\mathcal{Q}^\blambda$ is $\IC_{\psi_\nu(\blambda)}$. 

We have analogous statements for $M$. Let $Y^M_{[\nu]}=M\cdot(\cO_{[\nu]}+\fz_{L}^{M,\circ})$ be the variety analogous to $Y_{[\nu]}$ for the group $M$ and let $\cL^M_{[\nu]}$ be the local system $(\xymap^M_{[\nu]})_*\underline{\bk}$ coming from the map $\xymap^M_{[\nu]}:M\times^L(\cO_{[\nu]}+\fz_{L}^{M,\circ})\to Y^M_{[\nu]}$. Let $\cG^{\blambda,M}$ be the local system on $Y^M_{[\nu]}$ corresponding to $G^\blambda$. Then the indecomposable summands of $\cI^M_\nu$ are also indexed by $\uPart_\ell(\sm(\nu))$: if $\mathcal{Q}^{\blambda,M}$ denotes the indecomposable summand of $\cI^M_\nu$ such that $\bT_{\mathfrak{m}}((a_M)_!\mathcal{Q}^{\blambda,M})\cong\IC(Y^M_{[\nu]},\cG^{\blambda,M})$, then the head of $\mathcal{Q}^{\blambda,M}$ is $\cF^{\blambda,M}$.

Since $\Ind_{M \subset Q}^G$ is exact, to prove~\eqref{eqn:ind-Fnu} it suffices to prove that
\begin{equation}
\Ind_{M \subset Q}^G \mathcal{Q}^{\blambda,M}\cong \mathcal{Q}^\blambda\qquad\text{for all $\blambda\in\uPart_\ell(\sm(\nu))$.}
\end{equation} 
By Lemma~\ref{lem:transitivity} we have an isomorphism $\Ind_{M \subset Q}^G \cI^{M}_\nu\cong \cI_\nu$, and the indecomposable summands of $\cI^{M}_\nu$ and $\cI_\nu$ are indexed by the same set $\uPart_\ell(\sm(\nu))$, so (using Corollary~\ref{cor:induced-orbit}(2)) we know that $\Ind_{M \subset Q}^G \mathcal{Q}^{\blambda,M}\cong \mathcal{Q}^{f(\blambda)}$ for some permutation $f$ of $\uPart_\ell(\sm(\nu))$. We need to show that $f$ is the identity.

Via Corollary~\ref{cor:restriction-Fourier}, the isomorphism $\Ind_{M \subset Q}^G \cI^{M}_\nu\cong \cI_\nu$ induces an isomorphism on the Fourier transform side:
\begin{equation} \label{eqn:hard}
\uInd_{M \subset Q}^G\bigl(\IC(Y^M_{[\nu]},\cL^M_{[\nu]})\bigr)\cong\IC(Y_{[\nu]},\cL_{[\nu]}).
\end{equation}
Likewise, on the indecomposable summands we have $\uInd_{M\subset Q}^G(\IC(Y^M_{[\nu]},\cG^{\blambda,M}))\cong\IC(Y_{[\nu]},\cG^{f(\blambda)})$ for $f$ as above, and we need to show that $f$ is the identity.

We have already seen that the group $W_\nu$ acts on $\cL_{[\nu]}$ and hence on $\IC(Y_{[\nu]},\cL_{[\nu]})$, since $\cL_{[\nu]}$ is the local system on $Y_{[\nu]}$ corresponding to the regular representation of $W_\nu$. Similarly, it acts on $\IC(Y^M_{[\nu]},\cL^M_{[\nu]})$ and hence on the left-hand side of~\eqref{eqn:hard}. We claim that the isomorphism~\eqref{eqn:hard} is $W_\nu$-equivariant. This claim is closely analogous to a statement proved in~\cite[\S7.6]{ahr}, concerning the Weyl group equivariance of an induction isomorphism of Grothendieck sheaves. The triple $T\subset L\subset G$ of \emph{loc}.~\emph{cit}.\ is here replaced by $L\subset M\subset G$ (so that we have a relative Weyl group rather than a Weyl group), and the Galois coverings on the regular semisimple sets are here replaced by Galois coverings on more general strata of the Lusztig stratification, but it is easy to check that each step of the proof of equivariance can be carried out in this new context.

We therefore have a commutative diagram:
\[
\xymatrix{
\bk[W_\nu] \ar@{=}[rr] \ar[d]_-{\wr} & & \bk[W_\nu] \ar[d]^-{\wr} \\
\End \bigl( \IC(Y^M_{[\nu]},\cL^M_{[\nu]}) \bigr) \ar[r]^-{\uInd_{M \subset Q}^G} & \End \bigl( \uInd_{M \subset Q}^G\bigl(\IC(Y^M_{[\nu]},\cL^M_{[\nu]})\bigr) \bigr) \ar[r]^-{\sim} & \End \bigl( \IC(Y_{[\nu]},\cL_{[\nu]}) \bigr)
}
\]
and we conclude that the left-hand map on the bottom line is an isomorphism. Splitting $\IC(Y^M_{[\nu]},\cL^M_{[\nu]})$ and $\IC(Y_{[\nu]},\cL_{[\nu]})$ into their direct summands, it follows that the map
\begin{equation} \label{eqn:less-hard}
\Hom\bigl(\IC(Y^M_{[\nu]},\cL^M_{[\nu]}),\IC(Y^M_{[\nu]},\cG^{\blambda,M})\bigr)\to\Hom\bigl(\IC(Y_{[\nu]},\cL_{[\nu]}),\IC(Y_{[\nu]},\cG^{f(\blambda)})\bigr)
\end{equation}
induced by $\uInd_{M \subset Q}^G$ is also an isomorphism, and is $W_\nu$-equivariant if we define the $W_\nu$-action on each side using the $W_\nu$-actions on $\IC(Y^M_{[\nu]},\cL^M_{[\nu]})$ and $\IC(Y_{[\nu]},\cL_{[\nu]})$. But since $\IC$ is fully faithful, the left-hand side of~\eqref{eqn:less-hard} is isomorphic to 
\[ \Hom(\cL^M_{[\nu]},\cG^{\blambda,M})\cong \Hom_{\bk[W_\nu]}(\bk[W_\nu],G^\blambda)\cong G^\blambda, \] 
and the right-hand side is similarly isomorphic to $G^{f(\blambda)}$. So $f$ must be the identity, as desired.
\end{proof}

\section{Recollement}
\label{sec:recollement}

We retain the convention that $G = \GL(n)$ and that $\ell > 0$.  In this section, we will show that $\Perv_G(\cN_G,\bk)$ exhibits a `recollement' structure, generalizing~\cite[Corollary~5.2]{ahjr}.  (The meaning of `recollement' will be recalled following Theorem~\ref{thm:recollement}.)  In this section, for brevity, we will write $a$ for $a_G: \cN_G \hookrightarrow \fg$ and $\bT$ for $\bT_\fg: \Perv^{\mathrm{con}}_G(\fg,\bk) \to \Perv^{\mathrm{con}}_G(\fg,\bk)$.

\subsection{Notation and statement of the result}

Recall from \S\ref{subsect:statements-main} that $\fL$ denotes the set of conjugacy classes of Levi subgroups admitting a cuspidal pair.  We equip $\fL$ with a partial order by declaring that $\bL \leq \bL'$ if and only if there exist representatives $L \in \bL$ and $L' \in \bL'$ such that $L \subset L'$. We choose an enumeration $\bL_1, \ldots, \bL_r$ of $\fL$ such that 
\[
\bL_i \leq \bL_j \Longrightarrow i \leq j.
\]
(For example, when $n = 6$ and $\ell = 2$, one could enumerate the elements of $\fL$ in the order in which they are listed in Figure~\ref{fig}.)  In particular, $\bL_1$ is the class of maximal tori, and if $n$ is a power of $\ell$, then $\bL_r=\{G\}$.

For $i=1, \ldots, r$ we choose some representative $L_i$ for $\mathbf{L}_i$, and some parabolic subgroup $P_i \subset G$ containing $L_i$ as a Levi factor.  Let $W_i := N_G(L_i)/L_i$. Recall (see \S\ref{ss:combinatorics}) that this group is isomorphic to a product of symmetric groups.  Let $\Rep(W_i,\bk)$ denote the category of finite-dimensional $\bk[W_i]$-modules.

Next, let $Y_i = Y_{(L_i,\cO_i)}$ (see \S\ref{ss:induction-IC}), where $\cO_i$ is the principal nilpotent orbit in $\cN_{L_i}$.  Explicit descriptions of these subsets $Y_i$ were given in Remark~\ref{rmk:GL}.  Note that $\overline{Y_1} = \fg$, and that 
\begin{equation}
\bL_i \leq \bL_j \Longleftrightarrow \overline{Y_i} \supset Y_j.
\end{equation}
Let $\IC_i = \IC(\cO_i,\ubk)$.  This is the unique cuspidal simple object in $\Perv_{L_i}(\cN_{L_i},\bk)$.  Finally, let $\cI_i = \Ind_{L_i \subset P_i}^G \IC_i$. With this notation, Theorem \ref{thm:main} tells us that for each simple object $\cF$ of $\Perv_{G}(\cN_G,\bk)$ there exists a unique $i$ such that $\cI_i$ surjects to $\cF$.

Let $F_i = \bigcup_{j \ge i} \overline{Y_j}$, and let $\cA_i$ be the following full subcategory of $\Perv_G(\cN_G,\bk)$:
\[
\cA_i = \{ \cF \in \Perv_G(\cN_G,\bk) \mid \text{$\bT(a_!\cF)$ is supported on $F_i \subset \fg$} \}.
\]
In particular, $\cA_1 = \Perv_G(\cN_G,\bk)$.   We also let $\cA_{r+1}=\{0\}$. The main result of this section is the following.

\newcommand{\siq}{s}
\newcommand{\sil}{s^{\mathsf{L}}}
\newcommand{\sir}{s^{\mathsf{R}}}
\newcommand{\tsil}{\tilde{s}^{\mathsf{L}}}
\newcommand{\tsir}{\tilde{s}^{\mathsf{R}}}
\newcommand{\eiq}{e}
\newcommand{\eil}{e^{\mathsf{L}}}
\newcommand{\eir}{e^{\mathsf{R}}}
\newcommand{\teil}{\tilde{e}^{\mathsf{L}}}
\newcommand{\teir}{\tilde{e}^{\mathsf{R}}}

\begin{thm}
\label{thm:recollement}
For $i=1, \cdots, r$ there exists a recollement diagram
\[
\xymatrix@C=2.5cm{
\cA_{i+1} \ar[r] |{\textstyle\,\siq_i\,}& \cA_i \ar@/^15pt/[l]|{\textstyle\,\sir_i\,} \ar@/_15pt/[l]|{\textstyle\,\sil_i\,} \ar[r]|{\textstyle\,\eiq_i\,} & \Rep(W_i,\bk) \ar@/^15pt/[l]|{\textstyle\,\eir_i\,} \ar@/_15pt/[l]|{\textstyle\,\eil_i\,}
}
\]
such that $\siq_i$ is the natural inclusion.
\end{thm}

This means that the following properties hold~\cite{bbd,kuhn}:
\begin{enumerate}
\item $\siq_i$ is right adjoint to $\sil_i$ and left adjoint to $\sir_i$.
\item $\eiq_i$ is right adjoint to $\eil_i$ and left adjoint to $\eir_i$.
\item $\siq_i$ is fully faithful, and $\eiq_i(M) = 0$ if and only if $M \cong \siq_i(N)$ for some $N \in \cA_{i+1}$.
\item The adjunction morphisms $\eiq_i\eir_i \to \id$ and $\id \to \eiq_i\eil_i$ are isomorphisms.
\end{enumerate}

\begin{rmk}
Recall that by \cite{mautner} the category $\Perv_G(\cN_G,\bk)$ is equivalent to the category of finite dimensional modules over the Schur algebra $S_{\bk}(n,n)$. Hence our description of the category $\Perv_G(\cN_G,\bk)$ via successive recollements of categories $\Rep(W_i,\bk)$ can be considered as a geometric analogue of a similar result for the category of modules over $S_\bk(n,n)$, obtained as a special case of a more general result in \cite{kuhn} (see Examples 1.2, 1.5, 1.6, and 1.7 in \emph{loc}.~\emph{cit}., the last of which corresponds to the example illustrated in our Figure~\ref{fig}). We do not address the question of comparing these constructions.
\end{rmk}

\subsection{Preliminary results}

We begin with a few alternative descriptions of $\cA_i$.

\begin{lem}\label{lem:recolle-serre}
The following three subcategories of $\Perv_G(\cN_G,\bk)$ coincide:
\begin{enumerate}
\item the category $\cA_i$;
\item the Serre subcategory generated by the simple quotients of the $\cI_j$ for $j \ge i$;\label{it:serre-head}
\item the Serre subcategory generated by the $\cI_j$ for $j \ge i$.\label{it:serre-ind}
\end{enumerate}
\end{lem}
\begin{proof}
We begin by noting that $\cA_i$ is indeed a Serre subcategory of $\Perv_G(\cN_G,\bk)$.  This follows from the observation that in $\Perv_G(\fg,\bk)$, the property of being supported on the closed set $F_i$ is stable under extensions and subquotients.

Let $\cA_i'$ and $\cA_i''$ denote the categories described in~\eqref{it:serre-head} and~\eqref{it:serre-ind} above, respectively.  It is obvious that $\cA_i' \subset \cA_i''$.  Next, Proposition~\ref{prop:induction-IC} tells us that $\bT(a_!\cI_j)$ is supported on $\overline{Y_j}$, so $\cI_j \in \cA_i$ if $j \ge i$, and hence $\cA_i'' \subset \cA_i$.

Now consider a simple perverse sheaf $\cF \in \cA_i$.  By Theorem~\ref{thm:main}, there is a unique integer $k$ such that $\cF$ appears in the head of $\cI_k$. By Lemma \ref{lem:induction-cuspidal-new}, the restriction of $\bT(a_!\cF)$ to $Y_k$ is non zero, hence we must have $k \ge i$, so $\cF \in \cA_i'$.  We conclude that $\cA_i \subset \cA_i'$.
\end{proof}

\begin{lem}
\label{lem:ext1}
We have $\Ext^1_{\Perv_G(\cN_G,\bk)}(\IC_{(n)},\IC_{(n)}) = 0$.
\end{lem}

This lemma is an immediate consequence of~\cite[Theorem~4.1]{mautner} and general facts about modular representations of algebraic groups.  However, the following self-contained proof is perhaps easier to understand.

\begin{proof}
Let $j_{(n)} : \cO_{(n)} \hookrightarrow \cN_G$ denote the inclusion, and set
\[
\Delta'_{(n)}:= (j_{(n)})_! \underline{\bk}[\dim(\cO_{(n)})], \quad \Delta_{(n)}:={}^p (j_{(n)})_! \underline{\bk}[\dim(\cO_{(n)})].
\]
First we claim that
\begin{equation}
\label{eqn:ext-vanishing}
\Ext^1_{\Perv_G(\cN_G,\bk)}(\Delta_{(n)},\IC_{(n)}) = 0.
\end{equation}
By~\cite[Remarque~3.1.17(ii)]{bbd}, we have
\[
\Ext^1_{\Perv_{G}(\cN_G,\bk)}(\Delta_{(n)},\IC_{(n)}) = \Hom_{\Db_{G}(\cN_G,\bk)}(\Delta_{(n)},\IC_{(n)}[1]),
\]
so it is enough to prove that the right-hand side vanishes. Now consider the truncation triangle
\[
\mathcal{M} \to \Delta'_{(n)} \to \Delta_{(n)} \xrightarrow{+1}.
\]
As $\mathcal{M}$ is concentrated in perverse degrees${}\leq -1$, the induced morphism
\[
\Hom_{\Db_{G}(\cN_G,\bk)}(\Delta_{(n)},\IC_{(n)}[1]) \to \Hom_{\Db_{G}(\cN_G,\bk)}(\Delta'_{(n)},\IC_{(n)}[1])
\]
is injective, so it is enough to prove that the right-hand side vanishes. Now by adjunction we have
\begin{align*}
\Hom_{\Db_{G}(\cN_G,\bk)}(\Delta'_{(n)},\IC_{(n)}[1]) &\cong \Hom_{\Db_{G}(\cO_{(n)},\bk)}(\underline{\bk}[\dim(\cO_{(n)})], j_{(n)}^! \IC_{(n)}[1])\\
&\cong \mathsf{H}^1_G(\cO_{(n)},\bk) =0.
\end{align*}
For the last assertion, we have $\mathsf{H}^1_G(\cO_{(n)},\bk) \cong \mathsf{H}^1_{G_x}(\mathrm{pt},\bk)$ for any $x \in \cO_{(n)}$.  The equivariant cohomology group $\mathsf{H}^1_{G_x}(\mathrm{pt},\bk)$ vanishes because $G_x$ is the product of $\C^\times$ and a unipotent group.  This finishes the proof of \eqref{eqn:ext-vanishing}.

Now consider the exact sequence
\[
\ker \hookrightarrow \Delta_{(n)} \twoheadrightarrow \IC_{(n)},
\]
and the induced exact sequence
\[
\Hom(\ker,\IC_{(n)}) \to \Ext^1(\IC_{(n)},\IC_{(n)})
\to \Ext^1(\Delta_{(n)},\IC_{(n)}).
\]
We have checked in \eqref{eqn:ext-vanishing} that the third term in this sequence is trivial. The same is true for the first term since $\ker$ is supported on $\cN \smallsetminus \cO_{(n)}$. Hence the middle term is also $0$.
\end{proof}

\begin{prop}
\label{prop:projective}
The perverse sheaf $\cI_i$ is a projective object in $\cA_i$.
\end{prop}

\begin{proof}
By Lemma \ref{lem:recolle-serre}, $\cI_i$ is in $\cA_i$. Now
we must prove that the functor
\[
\Hom_{\Perv_G(\cN_G,\bk)}(\Ind^G_{L_i \subset P_i}(\IC_i), \cdot) \cong \Hom_{\Perv_{L_i}(\cN_{L_i},\bk)}(\IC_i, \Res^G_{L_i \subset P_i}(\cdot))
\]
is exact when restricted to $\cA_i$. We claim that any composition factor of a perverse sheaf of the form $\Res^G_{L_i \subset P_i}(\cF)$ for some $\cF \in \cA_i$ is isomorphic to $\IC_i$. This will imply the result in view of the exactness of the functor $\Res^G_{L_i \subset P_i}$ and Lemma~\ref{lem:ext1} (applied to $L_i$, which is a product of $\GL(m)$'s).

To prove the claim, we can assume that $\cF$ is simple. Since $\IC_i$ is the unique cuspidal simple perverse sheaf in $\Perv_{L_i}(\cN_{L_i},\bk)$, it is enough to prove that for any parabolic subgroup $Q \subsetneq L_i$ and any Levi factor $M$ of $Q$ we have
\[
\Res^{L_i}_{M \subset Q} \bigl( \Res^G_{L_i \subset P_i}(\cF) \bigr) = 0.
\]
However, by transitivity of restriction (Lemma~\ref{lem:transitivity}) we have
\[
\Res^{L_i}_{M \subset Q} \bigl( \Res^G_{L_i \subset P_i}(\cF) \bigr) \cong \Res^G_{M \subset QU_{P_i}}(\cF)
\]
where $U_{P_i}$ is the unipotent radical of $P_i$. Let us assume for a contradiction that $\Res^G_{M \subset QU_{P_i}}(\cF)$ is nonzero, so it has some simple subobject $\cF'$.  By adjunction, there is a nonzero morphism
\[
\Ind^G_{M \subset QU_{P_i}}(\cF') \to \cF,
\]
which must be surjective since $\cF$ is simple. By Corollary \ref{cor:series}, there exists a parabolic subgroup $Q' \subset Q \subsetneq L_i$, a Levi factor $M'$ of $Q'$ and a cuspidal simple perverse sheaf $\cF''$ on $\cN_{M'}$ such that $\cF'$ is a quotient of $\Ind^{M}_{M' \subset Q' \cap M}(\cF'')$. Then by exactness and transitivity of induction (see Lemma \ref{lem:transitivity}) we obtain a surjection
\[
\Ind^G_{M' \subset Q'U_{P_i}}(\cF'') \twoheadrightarrow \cF.
\]
But $M'$ belongs to some $G$-conjugacy class $\bL_j$ for some $j < i$.  By Lemma~\ref{lem:induction-cuspidal-new}, the existence of a surjection $\cI_j \to \cF$ contradicts the fact that $\cF$ is in $\cA_i$.
\end{proof}

\subsection{Proof of Theorem \ref{thm:recollement}}

First we define the functors $\eiq_i$. In fact, the functor $e_1$ is constructed in \cite{mautner}; the general construction will be very similar.

Choose a base point in the stratum $Y_i$. Then the fundamental group $B_i$ of the stratum for this choice of base point surjects to $W_i$. We denote by $j_i : Y_i \hookrightarrow \fg$ the inclusion, and consider the functor
\[
\alpha_i := j_i^* \circ \bT \circ a_! : \cA_i \to \Perv_G(Y_i,\bk).
\]
This functor is exact since $Y_i$ is open in the support of $\bT(a_!\cF)$ for any $\cF \in \cA_i$.

\begin{lem}
\label{lem:alpha_i}
For any $\cF$ in $\cA_i$, the object $\alpha_i(\cF)$ is a shifted $G$-equivariant local system on $Y_i$. Moreover, the corresponding representation of $B_i$ factors through $W_i$.
\end{lem}

\begin{proof}
Suppose first that $\cF$ is simple.  If $\cF$ actually lies in the smaller category $\cA_{i+1}$, then $\alpha_i(\cF) = 0$.  Otherwise, by Lemma~\ref{lem:recolle-serre}, $\cF$ is a quotient of $\cI_i$, and the claim follows from Lemma~\ref{lem:induction-cuspidal-new}\eqref{it:induction-head}.

For general $\cF$, the preceding paragraph already implies that $\alpha_i(\cF)$ is a shifted local system.  Now, choose $\cF' \subset \cF$ minimal with the property that $\alpha_i(\cF/\cF')=0$. Then by exactness $\alpha_i(\cF) \cong \alpha_i(\cF')$, so it is enough to prove the claim for $\cF'$. By minimality, no simple summand in the head of $\cF'$ is annihilated by $\alpha_i$.  Thus, every such simple summand lies in $\cA_i$, but not in $\cA_{i+1}$.  Invoking Lemma~\ref{lem:recolle-serre} again, we see that there is an $m \ge 1$ and a surjection $\cI_i^{\oplus m} \twoheadrightarrow \cF'/\rad \cF'$.  Proposition~\ref{prop:projective} then implies that this morphism lifts to a surjection $\cI_i^{\oplus m} \twoheadrightarrow \cF'$. By exactness, $\alpha_i(\cF')$ is a quotient of a direct sum of copies of $\alpha_i(\cI_i)$, which is the shifted local system corresponding to the regular representation of $W_i$.  The lemma follows.
\end{proof}

Note that the functors
\[
\Hom \bigl( (\varpi_{(L_i,\cO_i)})_* \underline{\bk}, \cdot \bigr) \quad \text{and} \quad \mathcal{L}_i'=\bigl( (\varpi_{(L_i,\cO_i)})_* \underline{\bk} \otimes \cdot \bigr)^{W_i}
\]
induce quasi-inverse equivalences of categories between the category of $\bk$-local systems on $Y_i$ such that the associated representation of $B_i$ factors through $W_i$, and the category $\Rep(W_i,\bk)$.
Using Lemma \ref{lem:alpha_i} we can define the exact functor
\[
\eiq_i : \cA_i \to \Rep(W_i,\bk), \quad \cF \mapsto \Hom \bigl( (\varpi_{(L_i,\cO_i)})_* \underline{\bk}, j_i^* \bT a_! \cF[-\dim (Y_i)] \bigr).
\]
We also define the functors
$
\teil_i, \teir_i : \Rep(W_i,\bk) \to \Perv_G(\cN_G, \bk)
$
by the formulas
\[
\teil_i := {}^p a^* \circ \bT^{-1} \circ {}^p (j_i)_! \circ \mathcal{L}_i, \quad \teir_i := {}^p a^! \circ \bT^{-1} \circ {}^p (j_i)_* \circ \mathcal{L}_i,
\]
where $\mathcal{L}_i(V)=\mathcal{L}_i'(V)[\dim (Y_i)]$ for any $V$ in $\Rep(W_i,\bk)$. (Note that these definitions are again direct generalizations of some constructions of \cite[\S 7]{mautner}.)

\begin{lem}
The functors $\teil_i$ and $\teir_i$ factor through functors
\[
\eil_i, \eir_i : \Rep(W_i,\bk) \to \cA_i.
\]
Moreover, $\eil_i$ is left adjoint to $\eiq_i$, and $\eir_i$ is right adjoint to $\eiq_i$.
\end{lem}

\begin{proof}
We prove both claims for $\teil_i$; the case of $\teir_i$ is similar. To prove the first claim it is enough to check that $\bT(a_! \teil_i(V))$ is supported on $\overline{Y_i}$ for any $V$ in $\Rep(W_i,\bk)$. But $a_! \teil_i(V)$ is a quotient of $\bT^{-1}({}^p (j_i)_!\mathcal{L}_i(V))$, so $\bT(a_!\teil_i(V))$ is a quotient of ${}^p (j_i)_! \mathcal{L}_i(V)$ in the category of perverse sheaves on $\fg$.  Thus, it is indeed supported on $\overline{Y_i}$. The adjointness statement is proved by the following computation:
\begin{align*}
\Hom_{\cA_i} \bigl( \eil_i(V),\cF \bigr) 
&\cong \Hom_{\Perv_G(\cN_G,\bk)} \bigl( {}^p a^* \bT^{-1}( {}^p (j_i)_! \mathcal{L}_i(V)), \cF \bigr) \\
&\cong \Hom_{\Perv_G(\fg,\bk)} \bigl( {}^p (j_i)_!  \mathcal{L}_i(V), \bT ( a_!\cF) \bigr) \\ 
&\cong \Hom_{\Perv_G(Y_i,\bk)} \bigl( \mathcal{L}_i(V), j_i^*  \bT ( a_!\cF )\bigr) \\
&\cong \Hom_{\Rep(W_i,\bk)} \bigl( V,\eiq_i(\cF) \bigr),
\end{align*}
which finishes the proof.
\end{proof}

\begin{lem}
The adjunction morphisms
\[
\eiq_i \circ \eir_i \to \mathrm{id} \quad \text{and} \quad \mathrm{id} \to \eiq_i \circ \eil_i 
\]
are isomorphisms.
\end{lem}
\begin{proof}
When $i = 1$, this statement is~\cite[Proposition~7.2]{mautner}.  In fact, the  argument given there establishes the result for general $i$.  We omit further details.
\end{proof}

Now we need to define the functors $\sil_i$, $\siq_i$ and $\sir_i$. The functor $\siq_i$ is simply the natural inclusion. To define $\sil_i$ and $\sir_i$, consider the closed inclusion
\[
k_i : F_i \hookrightarrow \fg.
\]
Then we define the functors
$
\tsil_i, \tsir_i : \cA_i \to \Perv_G(\cN_G,\bk)
$
by the formulas
\[
\tsil_i := {}^p a^* \circ \bT^{-1} \circ (k_{i+1})_* \circ {}^p (k_{i+1})^* \circ \bT \circ a_!, \quad \tsir_i := {}^p a^! \circ \bT^{-1} \circ (k_{i+1})_* \circ {}^p (k_{i+1})^! \circ \bT \circ a_!.
\]

\begin{lem}
The functors $\tsil_i$ and $\tsir_i$ factor through functors
\[
\sil_i, \sir_i : \cA_i \to \cA_{i+1}.
\]
Moreover, $\sil_i$ is left adjoint to $\siq_i$, and $\sir_i$ is right adjoint to $\siq_i$.
\end{lem}

\begin{proof}
We consider only the case of $\tsil_i$; the case of $\tsir_i$ is similar. We have to prove that if $\cF$ is in $\cA_i$, then
\[
\tsil_i(\cF) = {}^p a^*  \bT^{-1} ( (k_{i+1})_*  {}^p (k_{i+1})^* \bT ( a_!\cF))
\]
is in $\cA_{i+1}$, i.e.~$\bT ( a_! \tsil_i(\cF))$ is supported on $F_{i+1}$. However, $a_! \tsil_i(\cF)$ is a quotient of $\bT^{-1} (k_{i+1})_* {}^p (k_{i+1})^* \bT (a_!\cF)$, whose image under $\bT$ is clearly supported on  $F_{i+1}$. Hence the same is true for $a_! \tsil_i(\cF)$.

We now turn to the adjunction statement.  For $\cF$ in $\cA_i$ and $\cF'$ in $\cA_{i+1}$, we have:
\begin{align*}
\Hom_{\cA_{i+1}} \bigl( \sil_i(\cF),\cF' \bigr)
&= \Hom_{\cA_{i+1}} \bigl( {}^p a^* \bT^{-1} (k_{i+1})_* {}^p (k_{i+1})^* \bT (a_!\cF),\cF' \bigr) \\
&\cong \Hom_{\Perv_G(\fg,\bk)} \bigl( (k_{i+1})_* {}^p (k_{i+1})^* \bT (a_!\cF), \bT (a_!\cF') \bigr) \\ 
&\cong \Hom_{\Perv_G(F_{i+1},\bk)} \bigl( {}^p (k_{i+1})^* \bT (a_!\cF), {}^p (k_{i+1})^! \bT (a_!\cF') \bigr) \\
&\cong \Hom_{\Perv_G(\fg,\bk)} \bigl( \bT (a_!\cF), (k_{i+1})_* {}^p (k_{i+1})^! \bT (a_!\cF') \bigr).
\end{align*}
By definition of $\cF'$, the perverse sheaf $\bT (a_! \cF')$ is supported on $F_{i+1}$, so we have $(k_{i+1})_* {}^p (k_{i+1})^! \bT (a_!\cF') \cong \bT (a_!\cF')$. Hence we obtain
\[
\Hom \bigl( \bT (a_!\cF), (k_{i+1})_* {}^p (k_{i+1})^! \bT (a_!\cF') \bigr) \cong \Hom \bigl( \bT (a_!\cF), \bT (a_!\cF') \bigr) \cong \Hom(\cF,\cF'),
\]
which concludes the proof.
\end{proof}

We have checked conditions $(1)$, $(2)$ and $(4)$ in the definition of a recollement diagram. Condition $(3)$ is obvious from the definitions, so the proof of Theorem \ref{thm:recollement} is complete. \qed


\end{document}